\documentclass[a4paper,10pt]{article}

\usepackage[utf8]{inputenc}
\usepackage[english]{babel}
\usepackage{amsthm,amssymb,amsfonts,amsmath}
\usepackage[authoryear]{natbib}
\usepackage{dsfont}
\usepackage{color,fancybox,graphicx}
\usepackage{url}
\usepackage{psfrag}
\usepackage{color}
\usepackage{pifont}
\usepackage{latexsym,pstricks} 

\usepackage{hyperref}
\usepackage{mathrsfs }
\usepackage{pst-tree}
\usepackage{fancyhdr}
\usepackage{etoolbox}
\usepackage{setspace}
\usepackage{enumitem}
\usepackage[ ruled, linesnumbered]{algorithm2e}

\newcommand{\T}{\Theta}
\newcommand{\Pt}{\mathds{P}_{\Theta}}

\newcommand{\E}{\mathds{E}}

\renewcommand{\P}{\mathds{P}}

\newcommand{\bX}{{\bf X}}
\newcommand{\bx}{{\bf x}}
\newcommand{\bz}{{\bf z}}

\newtheorem{remark}{Remark}
\newtheorem{theorem}{Theorem}[section]

\newtheorem{lemme}{Lemma}
\newtheorem{techlemme}{Technical Lemma}

\newtheorem{corollary}{Corollary}

\newtheorem*{stone}{Stone's theorem \citeyearpar{St77}}

\newtheoremstyle{break}  
  {\topsep}   
  {\topsep}   
  {\itshape}  
  {0pt}       
  {\bfseries} 
  {}         
  {5pt plus 1pt minus 1pt}  
  {}          
\theoremstyle{break}
\newtheorem*{assumption}{}

\usepackage{natbib}
\bibpunct{(}{)}{;}{a}{,}{,}

\usepackage{hyperref}
\hypersetup{
  colorlinks = true,
  urlcolor = blue, 
  linkcolor = blue,
  citecolor = blue,
}

\begin{document}

\begin{center}
{\Large 
\textbf{\textsf{On the asymptotics of random forests}}}
\medskip
\medskip
\end{center}

\noindent{\bf Erwan Scornet }\\
{\it Sorbonne Universit\'es, UPMC Univ Paris 06, F-75005, Paris, France}\\
\href{mailto:erwan.scornet@upmc.fr}{erwan.scornet@upmc.fr}\

\medskip
\begin{abstract}
\noindent {\rm 
The last decade has witnessed a growing interest in random forest models which are recognized to exhibit good practical performance, especially in high-dimensional settings. On the theoretical side, however, their predictive power remains largely unexplained, thereby creating a gap between theory and practice. 
The aim of this paper is twofold. Firstly, we provide theoretical guarantees to link finite forests used in practice (with a finite number $M$ of trees) to their asymptotic counterparts (with $M = \infty$). Using empirical process theory, we prove a uniform central limit theorem for a large class of random forest estimates, which holds in particular for Breiman's original forests. 
Secondly, we show that infinite forest consistency implies finite forest consistency and thus, we state the consistency of several infinite forests. In particular, we prove that $q$ quantile forests---close in spirit to Breiman's forests but easier to study---are able to combine inconsistent trees to obtain a final consistent prediction, thus highlighting the benefits of random forests compared to single trees.

\medskip


\medskip

}
\end{abstract}

\section{Introduction}

Random forests are a class of algorithms used to solve classification and regression problems. As ensemble methods, they grow several trees as base estimates and aggregate them to make a prediction. In order to obtain many different trees based on a single training set, random forests procedures introduce randomness in the tree construction. For instance, trees can be built by randomizing the set of features \citep[][]{DiKo95, Ho98}, the data set \citep[][]{Br96,Br00}, or both at the same time \citep[][]{Br01,CuZh01}.

Among all random forest algorithms, the most popular one is that of Breiman \citeyearpar{Br01}, which relies on CART procedure \citep[Classification and Regression Trees,][]{BrFrOlSt84} to grow the individual trees.  As highlighted by many applied studies \citep[see, e.g.,][]{HaLa05, DiAl06}, Breiman's \citeyearpar{Br01} random forests often outperform state-of-the-art methods. They are recognized for their ability to handle high-dimensional data sets, thus being useful in fields such as genomics \citep[][]{Qi12} and pattern recognition \citep[][]{RoRiRaOrTo08}, just to name a few. On the computational side, Breiman's \citeyearpar{Br01} forests are easy to run and robust to changes in the parameters they depend on \citep[][]{LiWi02, GePoTu08}. As a proof of their success, many extensions have been developed in ranking problems \citep[][]{ClDeVa13}, quantile estimation \citep[][]{Me06}, and survival analysis \citep[][]{IsKoBlLa08}. Interesting new developments in the context of massive data sets have been achieved. For instance, \citet{GeErWe06} modified the procedure to reduce calculation time, while other authors extended the procedure to online settings \citep[][and the reference therein]{DeMaFr13, LaRoTe14}.

While Breiman's \citeyearpar{Br01} forests are extensively used in practice, some of their mathematical properties remain under active investigation. 
In fact, most theoretical studies focus on simplified versions of the algorithm, where the forest construction is independent of the training set. Consistency of such simplified models has been proved \citep[e.g.,][]{BiDeLu08,IsKo10, DeMaFr13}. However, these results do not extend  to Breiman's original forests whose construction critically depends on the whole training set. Recent attempts to bridge the gap between theoretical forest models and Breiman's \citeyearpar{Br01} forests have been made by \citet{Wa14} and \citet{ScBiVe14} who establish consistency of the original algorithm under suitable assumptions.

Apart from the dependence of the forest construction on the data set, there is another fundamental difference between existing forest models and ones implemented. Indeed, in practice, a forest can only be grown with a finite number $M$ of trees although most theoretical works assume, by convenience, that $M=\infty$. Since the predictor with $M=\infty$ does not depend on the specific tree realizations that form the forest, it is therefore more amenable to analysis. However, surprisingly, no study aims at clarifying the link between finite forests (finite $M$) and infinite forests ($M=\infty$) even if some authors  \citep[][]{MeHo14a,WaHaEf14} proved results on finite forest predictions at a fixed point $\bx$.

In the present paper, our goal is to study the connection between infinite forest models and finite forests used in practice in the context of regression. We start by proving a uniform central limit theorem for various random forests estimates, including Breiman's \citeyearpar{Br01} ones. In Section $3$, we also point out that the $\mathds{L}^2$ risk of infinite forests is lower than that of finite forests, which supports the interest of theoretical studies for infinite forests. Besides, this result shows that infinite forest consistency implies finite forest consistency.
Finally, in Section $4$, we prove the consistency of several infinite random forests. In particular, taking one step toward the understanding of Breiman's \citeyearpar{Br01} forests, we prove that $q$ quantile forests, a variety of forests whose construction depends on the positions $\bX_i$'s of the data, are consistent. As for Breiman's forests, each leaf of each tree in $q$ quantile forests contains a small number of points that does not grow to infinity with the sample size. Thus, $q$ quantile forests average inconsistent trees estimate to build a consistent prediction.

We start by giving some notation in Section $2$. All proofs are postponed to Section $5$.

\section{Notation}

Throughout the paper, we assume to be given a training sample $\mathcal D_n=(\bX_1,Y_1),$ $ \hdots, (\bX_n,Y_n)$ of $[0,1]^d\times$ $  \mathbb R$-valued independent random variables distributed as the prototype pair $(\bX,$ $ Y)$, where $\mathds{E}[Y^2]<\infty$. We aim at predicting the response $Y$, associated with the random variable $\bX$, by estimating the regression function $m(\bx) = \E \left[ Y | \bX = \bx\right]$. In this context, we use random forests to build an estimate $m_n: [0,1]^d \to \mathds{R}$ of $m$, based on the data set $\mathcal{D}_n$.

A random forest is a collection of $M$ randomized regression trees \citep[for an overview on tree construction, see Chapter $20$ in][]{GyKoKrWa02}. For the $j$-th tree in the family, the predicted value at point $\bx$ is denoted by $m_n(\bx, \Theta_j,\mathcal D_n)$, where $\Theta_1, \hdots,\Theta_M$ are independent random variables, distributed as a generic random variable $\Theta$, independent of the sample $\mathcal D_n$. 
This random variable can be used to sample the training set or to select the candidate directions or positions for splitting. The trees are combined to form the finite forest estimate 
\begin{align}
m_{M,n}({\bf x}, \Theta_1, \hdots, \Theta_M) = \frac{1}{M} \sum_{m=1}^M m_n(\bx, \Theta_m). \label{finite_forest}
\end{align}
By the law of large numbers, for any fixed $\bx$, conditionally on $\mathcal{D}_n$, the finite forest estimate tends to the infinite forest estimate 
\begin{align*}
m_{\infty,n}(\bx) = \E_{\Theta} \left[m_n(\bx, \Theta)\right].
\end{align*}
The risk of $m_{\infty,n}$ is defined by 
\begin{align}
R(m_{\infty,n}) = \mathds E [m_{\infty,n}(\bX)-m(\bX)]^2, \label{def_risk_infinite_RF}
\end{align}
while the risk of $m_{M,n}$ equals
\begin{align}
R(m_{M,n}) = \mathds E [m_{M,n}(\bX, \Theta_1, \hdots, \Theta_M)-m(\bX)]^2. \label{def_risk_finite_RF}
\end{align}
It is stressed that both risks $R(m_{\infty,n})$ and $R(m_{M,n})$ are deterministic since the expectation in (\ref{def_risk_infinite_RF})  is over $\bX, \mathcal{D}_n$, and the expectation in 
(\ref{def_risk_finite_RF}) is over $\bX, \mathcal{D}_n$ and $\Theta_1, \hdots, \Theta_M$. 
Throughout the paper, we say that $m_{\infty,n}$ (resp. $m_{M,n}$) is $\mathds{L}^2$ consistent if $R(m_{\infty,n})$ (resp. $R(m_{M,n})$) tends to zero as $n\to \infty$.

As mentioned earlier, there is a large variety of forests, depending on how trees are grown and how the randomness $\Theta$ influences the tree construction. 
For instance, tree construction can be independent of $\mathcal{D}_n$ \citep[][]{Bi12}, depend only on the $\bX_i$'s \citep[][]{BiDeLu08} or depend on the whole training set \citep[][]{CuZh01,GeErWe06, ZhZeKo12}.
Throughout the paper, we use Breiman's forests and uniform forests to exemplify our results. In Breiman's original procedure, splits depend on the whole sample and are performed to minimize variance within the two resulting cells. The algorithm stops when each cell contains less than a small pre-specified number of points (typically, $5$ in regression and $1$ in classification). On the other hand, uniform forests are a simpler procedure since, at each node, a coordinate is uniformly selected among $\{1, \hdots, d\}$ and a split position is uniformly chosen in the range of the cell, along the pre-chosen coordinate. The algorithm stops when a full binary tree of level $k$ is built, that is if each cell has been cut exactly $k$ times, where $k\in \mathds{N}$ is a parameter of the algorithm. 

%
%
%
%
%
%
%
%
%
%
%
%
%

In the rest of the paper, we will repeatedly use the random forest connection function $K_n$,  defined as 
\begin{center}
\begin{tabular}{cccl}
$K_n: $		& 	 $[0,1]^d  \times  [0,1]^d$     & $\to$     & $[0,1]$\\
			& 	 $(   \bx           ,     \bz    )$ & $\mapsto$ & $\P_{\T}\left[\bx \overset{\T}{\leftrightarrow} \bz \right]$,
\end{tabular}
\end{center}
where $\bx \overset{\T}{\leftrightarrow} \bz$ is the event where $\bx$ and $\bz$ belong to the same cell in the tree  $\mathcal{T}_n(\Theta)$ designed with $\Theta$ and $\mathcal{D}_n$. Moreover, notation $\mathds{P}_{\Theta}$ denotes the probability with respect to $\Theta$, conditionally on $\mathcal{D}_n$. The same notational convention holds for the expectation $\mathds{E}_{\Theta}$ and the variance $\mathds{V}_{\Theta}$. Thus, if we fix the training set $\mathcal{D}_n$, we see that the connection $K_n(\bx,\bz)$ is just the proportion of trees in which $\bx$ and $\bz$ are connected.

We say that a forest is discrete (resp. continuous) if, keeping $\mathcal{D}_n$ fixed, its connection function $K_n(\bullet, \bullet)$ is piecewise constant (resp. continuous).
In fact, most existing forest models fall in one of these two categories. 
For example, if, at each cell, the number of possible splits is finite, then the forest is discrete. This is the case of Breiman's forests, where splits can only be performed at the middle of two consecutive data points along any coordinate. However, if splits are drawn according to some density along each coordinate, the resulting forest is continuous. For instance, uniform forests are continuous.

\section{Finite and infinite random forests}

Contrary to finite forests which depend upon the particular $\Theta_j$'s used to design trees, infinite forests do not and are therefore more amenable to mathematical analysis. Besides, finite forests predictions can be difficult to interpret since they depend on the random parameters $\Theta_j$'s. In addition, the $\Theta_j$'s are independent of the data set and thus unrelated to the particular prediction problem.

%
%

In this section, we study the link between finite forests and infinite forests. More specifically, assuming that the data set $\mathcal{D}_n$ is fixed, we examine the asymptotic behavior of the finite forest estimate $m_{M,n}(\bullet, \Theta_1, \hdots, \Theta_M)$ as $M$ tends to infinity. This setting is consistent with practical problems, where the $\mathcal{D}_n$ is fixed, and one can grow as many trees as possible. 

Clearly, by the law of large numbers, we know that conditionally on $\mathcal{D}_n$, for all $\bx \in [0,1]^d$, almost surely, 
\begin{align}
m_{M,n}(\bx, \T_1, \hdots, \T_M) \underset{M \to \infty}{\to} m_{\infty,n}(\bx). \label{LLN_first}
\end{align}
The following theorem extend the pointwise convergence in (\ref{LLN_first}) to the convergence of the whole functional estimate $m_{M,n}(\bullet, \Theta_1, \hdots, \Theta_M)$, towards the functional estimate $m_{\infty, n}(\bullet)$.
\begin{theorem}
\label{almost_sure_convergence}
Consider a continuous or discrete random forest. Then, conditionally on $\mathcal{D}_n$, almost surely, for all $\bx \in [0,1]^d$, we have
\begin{align*}
m_{M,n}(\bx, \T_1, \hdots, \T_M) \underset{M \to \infty}{\to} m_{\infty,n}(\bx).
\end{align*}
\end{theorem}

\begin{remark}
Since the set $[0,1]^d$ is not countable, we cannot reverse the ``almost sure'' and ``for all $\bx \in [0,1]^d$'' statements in (\ref{LLN_first}). Thus, Theorem \ref{almost_sure_convergence} is not a consequence of (\ref{LLN_first}). 
\end{remark}

Theorem \ref{almost_sure_convergence} is a first step to prove that infinite forest estimates can be uniformly approximated by finite forest estimates. To pursue the analysis, a natural question is to determine the rate of convergence in Theorem \ref{almost_sure_convergence}. The pointwise rate of convergence is provided by the central limit theorem which says that, conditionally on $\mathcal{D}_n$, for all $\bx \in [0,1]^d$, 
\begin{align}
\sqrt{M} \big( m_{M,n}(\bx,\T_1, \hdots, \T_M) - m_{\infty,n}(\bx) \big) \overset{\mathcal{L}}{\underset{M \to \infty}{\rightarrow}}  \mathcal{N}\big(0, \tilde{\sigma}^2(\bx)\big), \label{TCL_first}
\end{align}
where $$\tilde{\sigma}^2(\bx) = \mathds{V}_{\Theta} \left( \frac{1}{N_n(\bx, \Theta)} \sum_{i=1}^n Y_i \mathds{1}_{\bx \overset{\Theta}{\leftrightarrow} \bX_i} \right) \leq 4 \max\limits_{1 \leq i \leq n} Y_i^2$$ 
and $N_n(\bx, \Theta)$ is the number of data points falling into the cell of the tree $\mathcal{T}_n(\T)$ which contains $\bx$.

Equation (\ref{TCL_first}) is not sufficient to determine the asymptotic distribution of the functional estimate $m_{M,n}(\bullet, \T_1, \hdots, \T_M)$. To make it explicit, we need to introduce the empirical process $\mathds{G}_M$ \citep[see][]{VaWe96} defined by 
\begin{align*}
\mathds{G}_M = \sqrt{M} \left( \frac{1}{M}\sum_{m=1}^M \delta_{\Theta_m} - \mathds{P}_{\Theta} \right),
\end{align*}
where $\delta_{\Theta_m}$ is the Dirac function at $\Theta_m$. We also let $\mathcal{F}_2 = \lbrace g_{\bx}: \theta \mapsto m_n(\bx, \theta); \bx \in$  $ [0,1]^d \rbrace$ be the collection of all possible tree estimates in the forest. In order to prove that a uniform central limit theorem holds for random forest estimates, we need to show that there exists a Gaussian process $\mathds{G}$ such that 
\begin{align}
\sup\limits_{g \in \mathcal{F}_2} \bigg\lbrace \int_{\Theta} |g(\theta)| \textrm{d}\mathds{G}_M(\theta) - \int_{\Theta} |g(\theta)| \textrm{d}\mathds{G}(\theta) \bigg\rbrace \underset{M \to \infty}{\to} 0,\label{statement_donsker}
\end{align}
where the first part on the left side can be written as 
\begin{align*}
\int_{\Theta} |g(\theta)| \textrm{d}\mathds{G}_M(\theta) = \sqrt{M} \left( \frac{1}{M}\sum_{m=1}^M |g(\Theta_m)| - \mathds{E}_{\Theta}\big[ |g(\Theta)|\big] \right).
\end{align*}
For more clarity, instead of (\ref{statement_donsker}), we will write
\begin{align}
\sqrt{M} \left( \frac{1}{M} \sum_{m=1}^M m_n(\bullet , \Theta_m)- \mathds{E}_{\Theta} \left[ m_n(\bullet, \Theta) \right] \right) \overset{\mathcal{L}}{\to} \mathds{G} g_{\bullet}. \label{functional_TCL_equation}
\end{align}
To establish identity (\ref{functional_TCL_equation}), we first define, for all $\varepsilon>0$, the random forest grid step $\delta(\varepsilon)$ by
\begin{align*}
\delta(\varepsilon) = \sup \left\lbrace \eta \in \mathds{R}:  \displaystyle \mathop{\sup_{\bx_1, \bx_2 \in [0,1]^d}}_{\|\bx_1 - \bx_2\|_{\infty} \leq \eta} \big|1 - K_n(\bx_1, \bx_2)\big| \leq \frac{\varepsilon^2}{8} \right\rbrace,
\end{align*}
where $K_n$ is the connection function of the forest. The function $\delta$ can be seen as the modulus of continuity of $K_n$ in the sense that it is the distance such that $K_n(\bx_1, \bx_2)$ does not vary of much that $\varepsilon^2/8$ if $\|\bx_1 - \bx_2\|_{\infty} \leq \delta(\varepsilon)$. We will also need the following assumption. 
\begin{assumption}
{\bf (H1)}
One of the following properties is satisfied:
\begin{itemize}
\item The random forest is discrete,
\item There exist $C,A >0$, $\alpha < 2$ such that, for all $\varepsilon>0$, 
\begin{align*}
\delta(\varepsilon) \geq C \exp(- A/\varepsilon^{\alpha}).
\end{align*}
\end{itemize}
\end{assumption}

Observe that {\bf (H1)} is mild since most forests are discrete and the only continuous forest we have in mind, the uniform forest, satisfies {\bf (H1)}, as stated in Lemma \ref{grid_step_uniform_forest} below. 

\begin{lemme} \label{grid_step_uniform_forest}
Let $k \in \mathds{N}$. Then, for all $\varepsilon>0$, the grid step $\delta(\varepsilon)$ of uniform forests of level $k$ satisfies
\begin{align*}
\delta(\varepsilon) \geq \exp\left( -\frac{A_{k,d}}{\varepsilon^{2/3}}\right),
\end{align*}
where $A_{k,d} = (8de(k+2)!)^{1/3}$. 
\end{lemme}

The following theorem states that a uniform central limit theorem is valid over the class of random forest estimates, providing that {\bf (H1)} is satisfied.

\begin{theorem}
\label{theo_convergence_distribution_foret}
Consider a random forest which satisfies {\bf (H1)}. Then,
\begin{align*}
\sqrt{M} \left( m_{M,n}(\bullet) - m_{\infty,n}(\bullet)  \right) \overset{\mathcal{L}}{\to} \mathds{G} g_{\bullet},
\end{align*}
where $\mathds{G}$ is a Gaussian process with mean zero and a covariate function
\begin{align*}
\emph{Cov}_{\Theta}(\mathds{G} g_{\bx}, \mathds{G} g_{\bz}) = \emph{Cov}_{\Theta} \left( \sum_{i=1}^n Y_i \frac{\mathds{1}_{\bx  \overset{\Theta}{\leftrightarrow} {\bf X}_i}}{N_n(\bx, \Theta)}, \sum_{i=1}^n Y_i \frac{\mathds{1}_{\bz  \overset{\Theta}{\leftrightarrow} {\bf X}_i}}{N_n(\bz, \Theta)} \right).
\end{align*}
\end{theorem}

According to the discussion above, Theorem \ref{theo_convergence_distribution_foret} holds for uniform forests (by Lemma \ref{grid_step_uniform_forest}) and Breiman's forests (since they are discrete). Moreover, according to this Theorem, the finite forest estimates tend uniformly to the infinite forest estimates, with the standard rate of convergence $\sqrt{M}$. This result contributes to bridge the gap between finite forests used in practice and infinite theoretical forests.

The proximity between two estimates can also be measured in terms of their $\mathds{L}^2$ risk. In this respect, Theorem \ref{lemme_dependency_M_n} states that the risk of infinite forests is lower than the one of finite forests and provides a bound on the difference between these two risks. We first need an assumption on the regression model.
\begin{assumption}{\bf (H2)}
One has 
\begin{align*}
Y = m(\bX) + \varepsilon,
\end{align*}
where $\varepsilon$ is a centered Gaussian noise with finite variance $\sigma^2$, independent of $\bX$, and $ \|m\|_{\infty} =  \sup\limits_{\bx \in [0,1]^d} |m(\bx)|$ $ < \infty.$
\end{assumption}

\begin{theorem} \label{lemme_dependency_M_n}
Assume that {\bf (H2)} is satisfied. Then, for all $M,n\in \mathds{N}^{\star}$, 
\begin{align*}
R(m_{M,n}) = R(m_{\infty,n}) + \frac{1}{M} \,\mathds{E}_{{\bf X}, \mathcal{D}_n} \Big[ \mathds{V}_{\Theta} \left[ m_n(\bX, \Theta) \right] \Big]. 
\end{align*}
In particular, 
\begin{align*}
0 \leq R(m_{M,n}) - R(m_{\infty,n}) & \leq \frac{8}{M} \times \big( \|m\|_{\infty}^2 +  \sigma^2 (1 + 4\log n) \big).
\end{align*}
\end{theorem}

Theorem \ref{lemme_dependency_M_n} reveals that the prediction accuracy of infinite forests is better than that of finite forests. In practice however, there is no simple way to implement infinite forests and, in fact, finite forests are nothing but Monte Carlo approximations of infinite forests. But, since the difference of risks between both types of forests is bounded (by Theorem \ref{lemme_dependency_M_n}), the prediction accuracy of finite forests is almost as good as that of infinite forests providing the number of trees is large enough. More precisely, under {\bf (H2)}, for all $\varepsilon>0$, if 
\begin{align*}
M \geq   \frac{8(\|m\|_{\infty}^2+\sigma^2)}{\varepsilon}  +  \frac{32\sigma^2 \log n}{\varepsilon},
\end{align*}
then $R(m_{M,n}) - R(m_{\infty,n}) \leq \varepsilon$.
%

Anoter interesting consequence of Theorem \ref{lemme_dependency_M_n} is that, assuming that {\bf (H2)} holds and that $M/\log n \to \infty$ as $n \to \infty$, finite random forests are consistent as soon as infinite random forests are. This alows to extend all previous consistency results regarding infinite forests \citep[see, e.g.,][]{Me06,BiDeLu08} to finite forests. It must be stressed that the ``$\log n$'' term comes from the Gaussian noise, since, if $\varepsilon_1,\hdots, \varepsilon_n$ are independent and distributed as a Gaussian noise $\varepsilon \sim \mathcal{N}(0, \sigma^2)$, we have, 
\begin{align*}
\E \left[ \max_{1\leq i \leq n} \varepsilon_i^2 \right] \leq \sigma^2(1 + 4 \log n),
\end{align*}
\citep[see, e.g., Chapter 1 in][]{BoLuMa13}. Therefore, the required number of trees depends on the noise in the regression model. For instance, if $Y$ is bounded, then the condition turns into $M \to \infty$.

\section{Consistency of some random forest models}

Section $3$ was devoted to the connection between finite and infinite forests. 
In particular, we proved in Theorem \ref{lemme_dependency_M_n} that the consistency of infinite forests implies that of finite forests, as soon as {\bf (H2)} is satisfied and $M/\log n \to \infty$. Thus, it is natural to focus on the consistency of infinite forest estimates, which can be written as
\begin{align}
m_{\infty,n }(\bX) = \sum_{i=1}^n W_{ni}^{\infty}(\bX) Y_i,\label{definition_foret_infinie}
\end{align}
where 
$$W_{ni}^{\infty}(\bX) = \E_{\T}\left[\frac{\mathds{1}_{\bX  \overset{\Theta}{\leftrightarrow} {\bf X}_i}}{N_n(\bX, \Theta)}\right]$$ 
are the random forest weights. 

Proving consistency of infinite random forests is in general a difficult task, mainly because forest construction can depend on both the $\bX_i$'s and the $Y_i$'s. This feature makes the resulting estimate highly data-dependent, and therefore difficult to analyze (this is particularly the case for Breiman's forests). To simplify the analysis, we investigate hereafter infinite random forest estimates whose weights depends only on $\bX, \bX_1, \hdots, \bX_n$ which is called the $X$-property. The good news is that when infinite forest estimates have the $X$-property, they fall in the general class of local averaging estimates, whose consistency can be addressed using Stone's \citeyearpar{St77} theorem.

%
%
%
Therefore, using Stone's theorem as a starting point, we first prove the consistency of random forests whose construction is independent of $\mathcal{D}_n$, which is the simplest case of  random forests satisfying the $X$-property. For such forests, the construction is based on the random parameter $\Theta$ only. As for now, we say that a forest is totally non adaptive of level $k$ ($k \in \mathds{N}$, with $k$ possibly depending on $n$) if each tree of the forest is built independently of the training set and if  each cell is cut exactly $k$ times. The resulting cell containing $\bX$, designed with randomness $\Theta$, is denoted by $A_n(\bX, \Theta)$.

\begin{theorem}\label{consistency_independent_forest}
Assume that $\bX$ is distributed on $[0,1]^d$ and consider a totally non adaptive forest of level $k$. In addition, assume that for all $\rho, \varepsilon >0$, there exists $N>0$ such that, with probability $1- \rho$, for all $n>N$,  
\begin{align*}
\textrm{diam}(A_n(\bX, \Theta)) \leq \varepsilon.
\end{align*}
Then, providing $k \to \infty$ and $2^{k}/n \to 0$, the infinite random forest is $\mathds{L}^2$ consistent, that is 
\begin{align*}
R(m_{\infty,n}) \to 0 \quad \textrm{as}~ n \to \infty.
\end{align*}

\end{theorem}

Theorem \ref{consistency_independent_forest} is a generalization of some consistency results in \citet{BiDeLu08} for the case of totally non adaptive random forest. Together with Theorem \ref{lemme_dependency_M_n}, we see that if {\bf (H2)} is satisfied and $M/\log n \to \infty$ as $n \to \infty$, then the finite random forest is $\mathds{L}^2$ consistent.

According to Theorem \ref{consistency_independent_forest}, a totally non adaptive forest of level $k$ is consistent if the cell diameters tend to zero as $n\to \infty$ and if the level $k$ is properly tuned. 
This is in particular true for uniform random forests, as shown in the following corollary.

\begin{corollary} \label{proposition_consistency_uniform_cut}
Assume that $\bX$ is distributed on $[0,1]^d$ and consider a uniform forest of level $k$. Then, providing that $k \to \infty$ and $2^{k}/n \to 0$, the uniform random forest is $\mathds{L}^2$ consistent. 
\end{corollary}

For totally non adaptive forests, the main difficulty that consists in using the data set to build the forest and to predict at the same time, vanishes. However, because of their simplified construction, these forests are far from accurately modelling Breiman's forest. To take one step further into the understanding of Breiman's \citeyearpar{Br01} forest behavior, we study the $q$ ($q \in [1/2,1)$) quantile random forest, which satisfies the $X$-property. Indeed, their construction depends on the $X_i$'s which is a good trade off between the complexity of Breiman's forests and the simplicity of totally non adaptive forests. As an example of $q$ quantile trees, the median tree ($q=1/2$) has already been studied by \citet{DeGyLu96}, such as the $k$-spacing tree \citep[][]{DeGyLu96} whose construction is based on quantiles. 

In the spirit of Breiman's algorithm, before growing each tree, data are subsampled, that is $a_n$ points ($a_n<n$) are selected without replacement. Then, each split is performed on an empirical $q_n$-quantile (where $q_n \in [1-q,q]$ can be pre-specified by the user or randomly chosen) along a coordinate, chosen uniformly at random among the $d$ coordinates. Recall that the $q'$-quantile ($q' \in [1-q,q]$) of $\bX_1, \hdots, \bX_n$ is defined as the only $\bX_{(\ell)}$ satisfying $F_n(\bX_{(\ell - 1)}) \leq q_n < F_n(\bX_{(\ell)})$, where the $\bX_{(i)}$'s are ordered increasingly. 
Note that data points on which splits are performed are not sent down to the resulting cells.
Finally, the algorithm stops when each cell contains exactly one point. The full procedure is described in {\bf Algorithm 2}.

\begin{center}
\begin{algorithm}[H]
 \KwIn{Fix $a_n \in \{1, \hdots, n \}$, and $\bx \in [0,1]^d$.}
 \KwData{A training set $\mathcal{D}_n$.}
 \For{$j=1, \hdots, M$}{
 Select $a_n$ points, without replacement, uniformly in $\mathcal{D}_n$.\
 
 Set $\mathcal{P} = \{[0,1]^p\}$ the partition associated with the root of the tree.\
 
  \While{there exists $A \in \mathcal{P}$ which contains strictly more than two points}{

Select uniformly one dimension $j$ within $\{1, \hdots, p\}$. \
 
Let $N$ be the number of data points in $A$ and select $q_n \in [1-q, q]\cap$ $ (1/N, 1-1/N)$.

Cut the cell $A$ at the position given by the $q_n$ empirical quantile (see definition below) along the $j$-th coordinate.

Call $A_L$ and $A_R$ the two resulting cell. \

Set $\mathcal{P} \leftarrow (\mathcal{P}\backslash\{A\})\cup A_L \cup A_R$.\
 
 }
 \For{each $A \in \mathcal{P}$ which contains exactly two points}{
Select uniformly one dimension $j$ within $\{1, \hdots, p\}$. \

Cut along the $j$-th direction, in the middle of the two points.
 
Call $A_L$ and $A_R$ the two resulting cell. \

Set $\mathcal{P} \leftarrow (\mathcal{P}\backslash\{A\})\cup A_L \cup A_R$.\

 }
 
 Compute the predicted value $m_n(\bx, \Theta_j)$ at $\bx$ equal to the single $Y_i$ falling in the cell of $\bx$, with respect to the partition $\mathcal{P}$.
 }
 Compute the random forest estimate $m_{M,n}(\bx; \Theta_1, \hdots, \Theta_M, \mathcal{D}_n)$ at the query point $\bx$ according to equality (\ref{finite_forest}).
 

\caption{$q$ quantile forest predicted value at $\bx$.}
\end{algorithm}
\end{center}

Since the construction of $q$ quantile forests depends on the $\bX_i$'s and is based on subsampling, it is a more realistic modeling of Breiman's forests than totally non adaptive forests. It also provides a good understanding on why random forests are still consistent even when there is exactly one data point in each leaf. Theorem \ref{theoreme_convergence_quantile_forest} states that with a proper subsampling rate of the training set, the $q$ quantile random forests are consistent. 

\begin{assumption}\label{assumption_last_th} {\bf (H3)}
One has 
\begin{align*}
Y = m(\bX) + \varepsilon,
\end{align*}
where $\varepsilon$ is a centred Gaussian noise with finite variance $\sigma^2$ variable, independent of $\bX$. Moreover, $\bX$ has a density bounded from below and from above and $m$ is continuous.
\end{assumption}

\begin{theorem}\
\label{theoreme_convergence_quantile_forest}
Assume that {\bf (H3)} is satisfied. Then, providing $a_n \rightarrow \infty$ et $a_n/n \rightarrow \infty$, the infinite $q$ quantile random forest is $\mathds{L}^2$ consistent.
\end{theorem}

Some remarks are in order. At first, observe that each tree in the $q$ quantile forest is inconsistent \citep[see Problem $4.3$ in][]{GyKoKrWa02}, because each leaf contains exactly one data point, a number which does not grow to infinity as $n \to \infty$. Thus, Theorem \ref{theoreme_convergence_quantile_forest} shows that $q$ quantile forest combines inconsistent trees to form a consistent estimate. 

Secondly, many random forests can be seen as quantile forests if they satisfy the $X$-property and if splits do not separate a small fraction of data points from the rest of the sample. The last assumption is true, for example, if $\bX$ has a density on $[0,1]^d$ bounded from below and from above, and if some splitting rule forces splits to be performed far away from the cell edges. This assumption is explicitly made in the analysis of \citet{Me06} and \citet{Wa14} to ensure that
cell diameters tend to zero as $n \to \infty$, which is a necessary condition to prove the consistency of partitioning estimates \citep[see Chapter $4$ in][]{GyKoKrWa02}.


We note finally that Theorem \ref{theoreme_convergence_quantile_forest} does not cover the bootstrap case since in that case, $a_n=n$ data points are selected with replacement. However, the condition on the subsampling rate can be replaced by the following one: for all $\bx$,
\begin{align}
\max_i \Pt \left[ \bx  \overset{\Theta}{\leftrightarrow} \bX_i \right] \to 0 ~\textrm{as}~n \to \infty. \label{proportion}
\end{align}
Condition (\ref{proportion}) can be interpreted by saying that a point $\bx$ should not be connected too often to the same data point in the forest, thus meaning that trees have to be various enough to ensure the forest consistency. This idea of diversity among trees has already been suggested by \citet{Br01}.
In bootstrap case, a single data point is selected in about $64\%$ of trees. Thus, the term $\max_i \Pt \left[ \bx  \overset{\Theta}{\leftrightarrow} \bX_i \right]$ is roughly upper bounded by $0.64$ which is not sufficient to prove (\ref{proportion}). It does not mean that random forests based on bootstrap are inconsistent but that a more detailed analysis is required. A possible, but probably difficult, route is an in-depth analysis of the connection function $K_n(\bx, \bX_i) = \Pt \left[ \bx  \overset{\Theta}{\leftrightarrow} \bX_i \right]$.

\section{Proofs}

\subsection{Proof of Theorem \ref{almost_sure_convergence}}

We assume that $\mathcal{D}_n$ is fixed and prove Theorem \ref{almost_sure_convergence} for $d=2$. The general case can be treated similarly. Throughout the proof, we write, for all $\theta$, $\bx, \bz \in [0,1]^2$,
\begin{align*}
f_{\bx, \bz }(\theta) = \frac{\mathds{1}_{\bx  \overset{\theta}{\leftrightarrow} \bz}}{N_n(\bx,\theta)}.
\end{align*}

Let us first consider a discrete random forest. By definition of such random forests, there exists $p \in \mathds{N}^{\star}$ and a partition $ \{ A_i: 1 \leq i \leq p \}$ of $[0,1]^2$ such that the connection function $K_n$ is constant over the sets $A_i\times A_j$'s ($1\leq i,j\leq p$). For all $1 \leq i \leq p$, denote by $a_i$, the center of the cell $A_i$. 
Take $\bx, \bz \in \mathds{R}^2$. There exist $i,j$ such that $\bx \in A_i, \bz \in A_j$. Thus, for all $\theta$,
\begin{align*}
\left|\frac{\mathds{1}_{\bx  \overset{\theta}{\leftrightarrow} \bz}}{N_n(\bx,\theta)} - \frac{\mathds{1}_{{\bf a}_i  \overset{\theta}{\leftrightarrow} {\bf a}_j}}{N_n({\bf a}_i, \theta)} \right| 
\leq & \left|\frac{\mathds{1}_{\bx  \overset{\theta}{\leftrightarrow} \bz}}{N_n(\bx,\theta)} - \frac{\mathds{1}_{{\bf a}_i  \overset{\theta}{\leftrightarrow} \bz}}{N_n({\bf a}_i,\theta)}  + \frac{\mathds{1}_{{\bf a}_i  \overset{\theta}{\leftrightarrow} \bz}}{N_n({\bf a}_i, \theta)}  - \frac{\mathds{1}_{{\bf a}_i  \overset{\theta}{\leftrightarrow} {\bf a}_j}}{N_n({\bf a}_i, \theta)} \right| \\
\leq & \,\frac{1}{{N_n({\bf a}_i,\theta)}} \left|\mathds{1}_{\bx  \overset{\theta}{\leftrightarrow} \bz} - \mathds{1}_{{\bf a}_i  \overset{\theta}{\leftrightarrow} \bz} \right| \\
&  +  \frac{1}{N_n({\bf a}_i, \theta)} \left| \mathds{1}_{{\bf a}_i  \overset{\theta}{\leftrightarrow} \bz}  - \mathds{1}_{{\bf a}_i  \overset{\theta}{\leftrightarrow} {\bf a}_j} \right| \\
 \leq &  \,\frac{1}{{N_n({\bf a}_i,\theta)}} \mathds{1}_{\bx  \overset{\theta}{\nleftrightarrow} {\bf a}_i} +   \frac{1}{N_n({\bf a}_i, \theta)} \mathds{1}_{{\bf a}_j  \overset{\theta}{\nleftrightarrow} \bz}  \\
 \leq & ~ 0.
\end{align*}
Thus, the set 
\begin{align*}
\mathcal{H} = \left\lbrace \theta \mapsto f_{\bx, \bz }(\theta): \bx, \bz  \in [0,1]^2 \right\rbrace
\end{align*}
is finite. Therefore, by the strong law of large numbers, almost surely, for all $f \in \mathcal{H}$, 
\begin{align*}
\frac{1}{M} \sum_{m=1}^M f(\Theta_m) \underset{M \to \infty}{\to} \mathds{E}_{\Theta}\big[ f(\Theta) \big].
\end{align*}
Noticing that $W_{ni}^M(\bx) = \frac{1}{M} \sum_{m=1}^M f_{\bx, \bX_i}(\Theta_m)$, we obtain that, almost surely, for all $\bx \in [0,1]^2$,
\begin{align*}
 W_{ni}^M(\bx) \to W_{ni}^{\infty}(\bx), \quad \textrm{as}~M \to \infty.
\end{align*} 
Since $\mathcal{D}_n$ is fixed and random forest estimates are linear in the weights, the proof of the discrete case is complete. 

Let us now consider a continuous random forest. We define, for all ${\bf x},{\bf z} \in [0,1]^2$, 
\begin{align*}
W^{M}_n(\bx,\bz) = \frac{1}{M} \sum_{m=1}^M \frac{\mathds{1}_{\bx  \overset{\Theta_m}{\leftrightarrow} \bz}}{N_n(\bx,\Theta_m)},  
\end{align*}
and
\begin{align*}
W^{\infty}_n(\bx,\bz) = \mathds{E}_{\Theta}\left[ \frac{\mathds{1}_{\bx  \overset{\Theta}{\leftrightarrow} \bz}}{N_n(\bx,\Theta)} \right]. 
\end{align*}

According to the strong law of large numbers, almost surely, for all $\bx,\bz \in [0,1]^2\cap\mathbb{Q}^2$, 
\begin{align*}
\lim\limits_{M \to \infty} W^{M}_n(\bx,\bz) = W^{\infty}_n(\bx,\bz). 
\end{align*}
Set $\bx,\bz \in [0,1]^2$ where $\bx = (x^{(1)},x^{(2)})$ and $\bz = (z^{(1)},z^{(2)})$. Assume, without loss of generality, that  $x^{(1)} < z^{(1)}$ and $x^{(2)} < z^{(2)}$. Let 
\begin{align*}
& A_{\bx} = \{ {\bf u} \in [0,1]^2, u^{(1)} \leq x^{(1)} ~\textrm{and}~ u^{(2)} \leq x^{(2)} \}, \\
\textrm{and}~& A_{\bz} = \{ {\bf u} \in [0,1]^2, u^{(1)} \geq z^{(1)} ~\textrm{and}~ u^{(2)} \geq z^{(2)} \}.
\end{align*}
Choose $\bx_1 \in A_{\bx} \cap \mathbb{Q}^2$ (resp. $\bz_2 \in A_{\bz} \cap \mathbb{Q}^2$) and take $\bx_2 \in [0,1]^2\cap \mathbb{Q}^2$ (resp. $\bz_1 \in [0,1]^2\cap \mathbb{Q}^2$) such that $\bx_1, \bx, \bx_2$ (resp. $\bz_2, \bz, \bz_1$) are aligned in this order (see Figure \ref{figure_1b}).
\begin{figure}[h!]
\begin{center}
\includegraphics[scale=0.5]{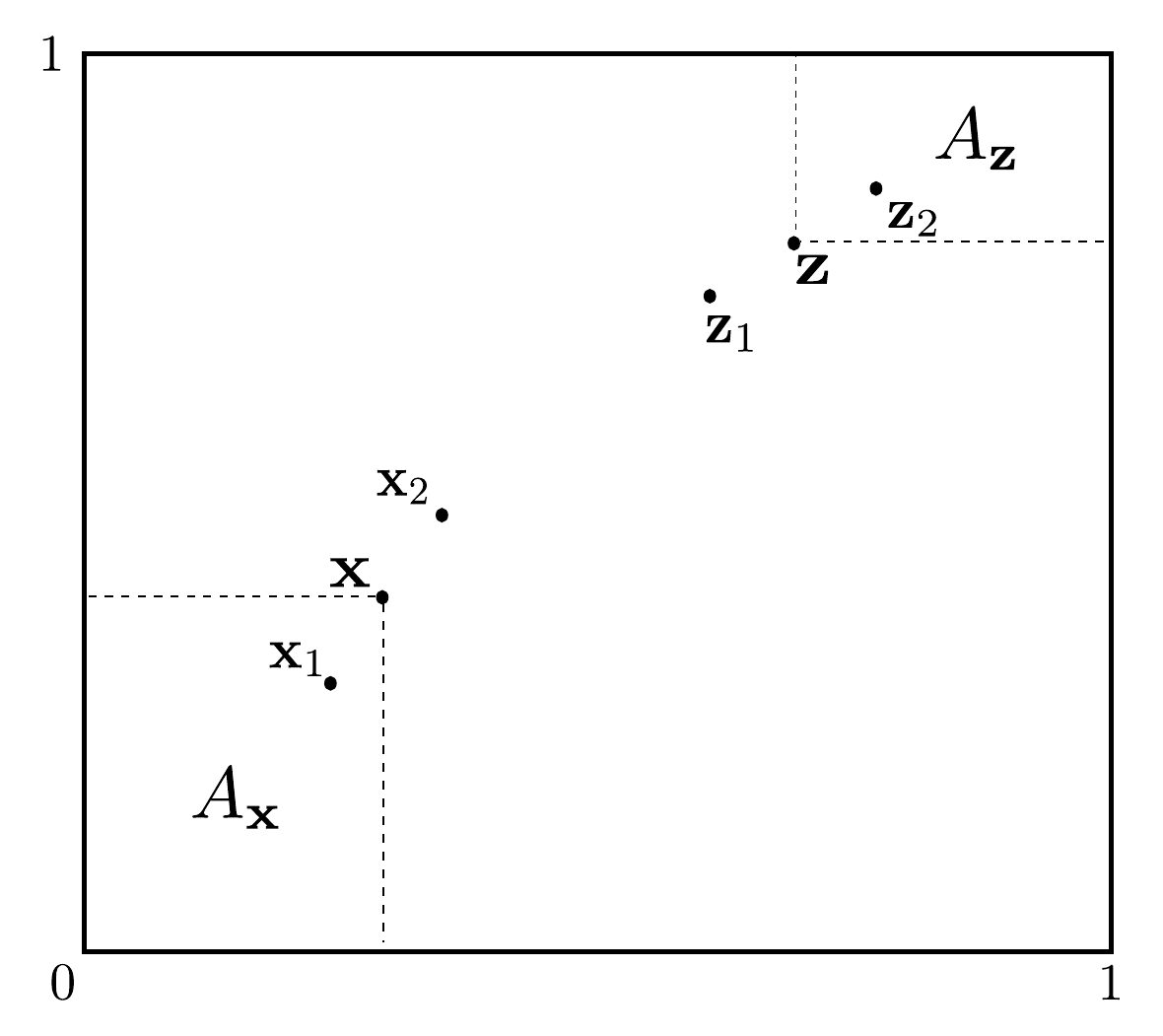}   
\end{center}
\caption{Respective positions of $\bx, \bx_1, \bx_2$ and $\bz, \bz_1, \bz_2$}
\label{figure_1b}
\end{figure}

Thus,
\begin{align}
\left| W^M_n(\bx,\bz) - W^{\infty}_n(\bx,\bz) \right| \leq & \left| W^M_n(\bx,\bz) - W^M_n(\bx_1, \bz_2) \right| \nonumber\\
& + \left| W^M_n(\bx_1,\bz_2)  - W^{\infty}_n(\bx_1,\bz_2) \right|  \nonumber \\
& + \left| W^{\infty}_n(\bx_1,\bz_2) - W^{\infty}_n(\bx,\bz) \right|. \label{lemmeforet3}
\end{align}
Set $\varepsilon > 0$. Because of the continuity of $K_n$, we can choose $\bx_1, \bx_2$ close enough to $\bx$ and $\bz_2, \bz_1$ close enough to $\bz$ such that,
\begin{align*}
 |K_n(\bx_2, \bx_1)-1|  & \leq \varepsilon,  \\
 |K_n(\bz_1, \bz_2)-1| & \leq \varepsilon,  \\
 |1 - K_n(\bx_1, \bx)| & \leq \varepsilon, \\
  |1 - K_n(\bz_2, \bz)| & \leq \varepsilon.  
\end{align*}
%
%
Let us consider the second term in equation (\ref{lemmeforet3}). Since $\bx_1, \bz_2$ belong to $[0,1]^2 \cap \mathbb{Q}^2$, almost surely, there exists $M_1>0$ such that, if $M > M_1$,
\begin{align*}
\left| W^M_n(\bx_1,\bz_2) - W^{\infty}_n(\bx_1,\bz_2) \right| \leq \varepsilon. 
\end{align*}

Considering the first term in (\ref{lemmeforet3}), we have
\begin{align*}
\left| W^M_n(\bx,\bz) - W^{\infty}_n(\bx_1,\bz_2) \right| \leq \frac{1}{M}\sum_{m=1}^M \left| \frac{\mathds{1}_{\bx \overset{\Theta_m}{\leftrightarrow} \bz}}{N_n(\bx,\Theta_m)} - \frac{\mathds{1}_{\bx_1 \overset{\Theta_m}{\leftrightarrow} \bz_2}}{N_n(\bx,\Theta_m)} \right|. 
\end{align*}
Observe that, given the positions of $\bx, \bx_1, \bz, \bz_2$, the only case where 
\begin{align*}
\left| \frac{\mathds{1}_{\bx \overset{\Theta_m}{\leftrightarrow} \bz}}{N_n(\bx,\Theta_m)} - \frac{\mathds{1}_{\bx_1 \overset{\Theta_m}{\leftrightarrow} \bz_2}}{N_n(\bx,\Theta_m)} \right| \neq 0 
\end{align*}
occurs when $\bx_1  \overset{\Theta_m}{\nleftrightarrow} \bz_2$ and $\bx  \overset{\Theta_m}{\leftrightarrow} \bz$.  Thus,
\begin{align*}
 & \frac{1}{M}\sum_{m=1}^M \left| \frac{\mathds{1}_{\bx \overset{\Theta_m}{\leftrightarrow} \bz}}{N_n(\bx,\Theta_m)} - \frac{\mathds{1}_{\bx_1 \overset{\Theta_m}{\leftrightarrow} \bz_2}}{N_n(\bx,\Theta_m)} \right| \\
 & \quad = \frac{1}{M}\sum_{m=1}^M \left| \frac{\mathds{1}_{\bx \overset{\Theta_m}{\leftrightarrow} \bz}}{N_n(\bx,\Theta_m)} - \frac{\mathds{1}_{\bx_1 \overset{\Theta_m}{\leftrightarrow} \bz_2}}{N_n(\bx,\Theta_m)} \right| \mathds{1}_{\bx_1  \overset{\Theta_m}{\nleftrightarrow} \bz_2} \mathds{1}_{\bx  \overset{\Theta_m}{\leftrightarrow} \bz} \\
& \quad \leq \frac{1}{M}\sum_{m=1}^M \mathds{1}_{\bx \overset{\Theta_m}{\leftrightarrow} \bz} \mathds{1}_{\bx_1  \overset{\Theta_m}{\nleftrightarrow} \bz_2}.
\end{align*}

Again, given the relative positions of $\bx, \bx_1, \bx_2, \bz, \bz_2,\bz_1$, we obtain  
\begin{align*}
 \frac{1}{M}\sum_{m=1}^M \mathds{1}_{\bx \overset{\Theta_m}{\leftrightarrow} \bz} \mathds{1}_{\bx_1  \overset{\Theta_m}{\nleftrightarrow} \bz_2} & \leq  \frac{1}{M}\sum_{m=1}^M \left( \mathds{1}_{\bx_1 \overset{\Theta_m}{\nleftrightarrow} \bx } + \mathds{1}_{\bz_2 \overset{\Theta_m}{\nleftrightarrow} \bz } \right)\\
&  \leq   \frac{1}{M}\sum_{m=1}^M \left( \mathds{1}_{\bx_1 \overset{\Theta_m}{\nleftrightarrow} \bx_2 } + \mathds{1}_{\bz_2 \overset{\Theta_m}{\nleftrightarrow} \bz_1 } \right) \\
&  \leq  \frac{1}{M}\sum_{m=1}^M \mathds{1}_{\bx_1 \overset{\Theta_m}{\nleftrightarrow} \bx_2 } + \frac{1}{M}\sum_{m=1}^M \mathds{1}_{\bz_2 \overset{\Theta_m}{\nleftrightarrow} \bz_1 } .
\end{align*}
Collecting the previous inequalities, we have
\begin{align*}
\left| W^M_n(\bx,\bz) - W^{\infty}_n(\bx_1,\bz_2) \right|  \leq & ~\frac{1}{M}\sum_{m=1}^M \mathds{1}_{\bx_1 \overset{\Theta_m}{\nleftrightarrow} \bx_2 } 
+  \frac{1}{M}\sum_{m=1}^M \mathds{1}_{\bz_2 \overset{\Theta_m}{\nleftrightarrow} \bz_1 } \\
\leq & ~ 2 - \frac{1}{M}\sum_{m=1}^M \mathds{1}_{\bx_1 \overset{\Theta_m}{\leftrightarrow} \bx_2 } 
 - \frac{1}{M}\sum_{m=1}^M \mathds{1}_{\bz_2 \overset{\Theta_m}{\leftrightarrow} \bz_1 }. 
\end{align*}
Since $\bx_2, \bz_1, \bx_1, \bz_2 \in [0,1]^2\cap \mathbb{Q}^2$, we deduce that there exists $M_2$ such that, for all $M > M_2$, 
\begin{align}
\left| W^M_n(\bx,\bz) - W^{\infty}_n(\bx_1,\bz_2) \right|  \leq 2 - K_{\infty}(\bx_2, \bx_1) - K_{\infty}(\bz_1, \bz_2) + 2 \varepsilon. \label{inequality1_th1}
\end{align}
Considering the third term in (\ref{lemmeforet3}), using the same arguments as above, we see that 
\begin{align}
| W^{\infty}_n(\bx_1, \bz_2) - W^{\infty}_n(\bx, \bz)| & \leq \mathds{E}_{\Theta} \left| \frac{\mathds{1}_{\bx_1  \overset{\Theta}{\leftrightarrow} \bz_2}}{N_n(\bx_1,\Theta)}  - \frac{\mathds{1}_{\bx  \overset{\Theta}{\leftrightarrow} \bz}}{N_n(\bx,\Theta)} \right|\nonumber\\
& \leq \mathds{E}_{\Theta} \left| \frac{\mathds{1}_{\bx_1  \overset{\Theta}{\leftrightarrow} \bz_2}}{N_n(\bx_1,\Theta)}  - \frac{\mathds{1}_{\bx  \overset{\Theta}{\leftrightarrow} \bz}}{N_n(\bx,\Theta)} \mathds{1}_{\bx_1  \overset{\Theta}{\nleftrightarrow} \bz_2} \mathds{1}_{\bx  \overset{\Theta}{\leftrightarrow} \bz} \right|\nonumber\\
& \leq \mathds{E}_{\Theta} \left[ \mathds{1}_{\bx_1  \overset{\Theta}{\nleftrightarrow} \bz_2} \mathds{1}_{\bx  \overset{\Theta}{\leftrightarrow} \bz} \right]\nonumber\\
& \leq \mathds{E}_{\Theta} \left[ \mathds{1}_{\bx_1 \overset{\Theta}{\nleftrightarrow} \bx_2 } 
+  \mathds{1}_{\bz_2 \overset{\Theta}{\nleftrightarrow} \bz_1 } \right] 
 \nonumber \\
& \leq 2-K_n(\bx_1, \bx_2) - K_n(\bz_2,\bz_1 ).\label{inequality2_th1}
\end{align}
Using inequalities (\ref{inequality1_th1}) and (\ref{inequality2_th1}) in (\ref{lemmeforet3}), we finally conclude that, for all $M > \max(M_1, M_2)$,  
\begin{align*}
 \left| W^{M}_n(\bx,\bz) - W^{\infty}_n(\bx,\bz) \right| & \leq 4 - 2K_{\infty}(\bx_2, \bx_1) - 2K_{\infty}(\bz_1, \bz_2) + 3 \varepsilon\\
 & \leq 7 \varepsilon.
\end{align*}
This completes the proof of Theorem \ref{almost_sure_convergence}.

\subsection{Proof of Lemma \ref{grid_step_uniform_forest} and Theorem \ref{theo_convergence_distribution_foret}}

\begin{proof}[Proof of Lemma \ref{grid_step_uniform_forest}]

Set $k \in \mathds{N}$ and $\varepsilon > 0$. We start by considering the case where $d=1$. Take  $x,z \in [0,1]$ and let $w = - \log\left( |x-z|\right)$. The probability that $x$ and $z$ are not connected in the uniform forest after $k$ cuts is given by
\begin{align*}
1 - K_k(x,z) & \leq 1 - K_k(0, |z-x|)\\
& \quad \quad \textrm{(according to Technical Lemma \ref{tech_lemma_1}, see the end of the section)}\\
& \leq  e^{-w} \mathds{1}_{k>0}\sum_{i=0}^{k-1} \frac{w^i}{i!}\\
& \quad \quad \textrm{(according to Technical Lemma \ref{lemme_connection_probability_uniform_forest_dim1}, see the end of the section)}\\
& \leq \frac{(k+2)!e}{w^{3}},
\end{align*}
for all $w>1$. Now, consider the multivariate case, and let $\bx, \bz \in [0,1]^d$. Set, for all $1 \leq j \leq d$, $w_j = - \log\left( |x_j-z_j|\right)$. By union bound, recalling that $1 - K_k(\bx,\bz) = \mathds{P}_{\Theta}(\bx \overset{\Theta}{\nleftrightarrow} \bz )$, we have
\begin{align*}
1 - K_k(\bx,\bz) & \leq  \sum_{j=1}^d  \left( 1 - K_k(x_j,z_j)\right) \\
& \leq \frac{d(k+2)!e}{ \min\limits_{1 \leq j \leq d}w_j^{3}}.
\end{align*}
Thus, if, for all $1 \leq j \leq d$, 
\begin{align*}
|x_j - z_j| \leq \exp\left( -\frac{(A_{k,d})^{1/3}}{\varepsilon^{2/3}}\right), 
\end{align*}
then
\begin{align*}
1 - K_k(\bx,\bz) \leq \frac{\varepsilon^2}{8},
\end{align*}
where $A_{k,d} = (8de(k+2)!)^{1/3}$. Consequently, 
\begin{align*}
\delta(\varepsilon) \geq \exp\left(-\frac{(A_{k,d})^{1/3}}{\varepsilon^{2/3}}\right).
\end{align*}

\end{proof}

\begin{proof}[Proof of Theorem \ref{theo_convergence_distribution_foret}]

We start the proof by proving that the class 
\begin{align*}
\mathcal{H} = \left\lbrace \theta \mapsto f_{\bx, \bz }(\theta): \bx, \bz  \in \mathds{R}^2 \right\rbrace
\end{align*}
is $\Pt$-Donsker, that is, there exists a Gaussian process $\mathds{G}$ such that
\begin{align*}
\sup\limits_{f \in \mathcal{H}} \big\lbrace \mathds{E}|f|\left(d\mathds{G}_M - d\mathds{G} \right) \big\rbrace \underset{M \to \infty}{\to} 0.
\end{align*}

At first, let us consider a finite random forest. As noticed in the proof of Theorem \ref{almost_sure_convergence}, the set $\mathcal{H}$ is finite. Consequently, by the central limit theorem, the set $\mathcal{H}$ is $\Pt$-Donsker.

Now, consider a random forest which satisfies the second statement in {\bf Assumption 1}. Set $\varepsilon > 0$. Consider a regular grid of $[0,1]^d$ with a step $\delta$ and let $\mathcal{G}_{\delta}$ be the set of nodes of this grid. We start by finding a condition on $\delta$ such that the set 
\begin{align*}
\tilde{\mathcal{G}}_{\delta} = \left\lbrace [f_{\bx_1, \bz_1},f_{\bx_2, \bz_2}]: \bx_1, \bx_2, \bz_1, \bz_2 \in \mathcal{G}_{\delta} \right\rbrace
\end{align*}
is a covering of $\varepsilon$-bracket of the set $\mathcal{H}$, that is, for all $f\in \mathcal{H}$, there exists $\bx_1, \bz_1, \bx_2, \bz_2 \in \mathcal{G}_{\delta}$ such that 
\begin{align}
f_{\bx_1, \bz_1} \leq f \leq f_{\bx_2, \bz_2} ~\textrm{and}~  \E^{1/2} \left[ f_{\bx_2, \bz_2}(\Theta) - f_{\bx_1, \bz_1}(\Theta) \right]^2 \leq \varepsilon. \label{condition_bracketing}
\end{align}

To this aim, set $\bx,\bz \in [0,1]^d$ and choose $\bx_1, \bx_2, \bz_1, \bz_2 \in \mathcal{G}_{\delta}$ (see Figure \ref{figure_proof2}). Note that, for all $\theta$, 
\begin{align*}
\frac{\mathds{1}_{\bx_1  \overset{\theta}{\leftrightarrow} \bz_2}}{N_n(\bx_1, \theta)} \leq \frac{\mathds{1}_{\bx  \overset{\theta}{\leftrightarrow} \bz}}{N_n(\bx, \theta)} \leq \frac{\mathds{1}_{\bx_2  \overset{\theta}{\leftrightarrow} \bz_1}}{N_n(\bx_2, \theta)},
\end{align*}
that is, $f_{\bx_1, \bz_2} \leq f_{\bx, \bz} \leq f_{\bx_2, \bz_1}$. To prove the second statement in (\ref{condition_bracketing}), observe that
\begin{align*}
\E^{1/2} \big[ f_{\bx_2, \bz_2}(\Theta) - f_{\bx_1, \bz_1}(\Theta) \big]^2 & = \mathds{E}_{\Theta}^{1/2} \left[ \frac{\mathds{1}_{\bx_1 \overset{\Theta}{\leftrightarrow} \bz_2}}{N_n(\bx_1, \Theta)} - \frac{\mathds{1}_{\bx_2  \overset{\Theta}{\leftrightarrow} \bz_1}}{N_n(\bx_2, \Theta)} \right]^2\\
& = \mathds{E}_{\Theta}^{1/2} \bigg[ \bigg( \frac{\mathds{1}_{\bx_1 \overset{\Theta}{\leftrightarrow} \bz_2}}{N_n(\bx_1, \Theta)} - \frac{\mathds{1}_{\bx_2  \overset{\Theta}{\leftrightarrow} \bz_1}}{N_n(\bx_2, \Theta)} \bigg) \\
& \qquad \qquad    \times \mathds{1}_{\bx_1  \overset{\Theta}{\nleftrightarrow} \bz_2} \mathds{1}_{\bx_2  \overset{\Theta}{\leftrightarrow} \bz_1} \bigg]^2\\
& \leq  \mathds{E}_{\Theta}^{1/2} \left[ \mathds{1}_{\bx_1 \overset{\Theta}{\nleftrightarrow} \bx_2 } 
+  \mathds{1}_{\bz_1 \overset{\Theta}{\nleftrightarrow} \bz_2 }  \right]^2\\
& \leq 2 \sqrt{ 1 - K_n(\bx_1, \bx_2) + 1 - K_n(\bz_1, \bz_2)}.
\end{align*}

\begin{figure}[h!]
\begin{center}
\includegraphics[scale=0.5]{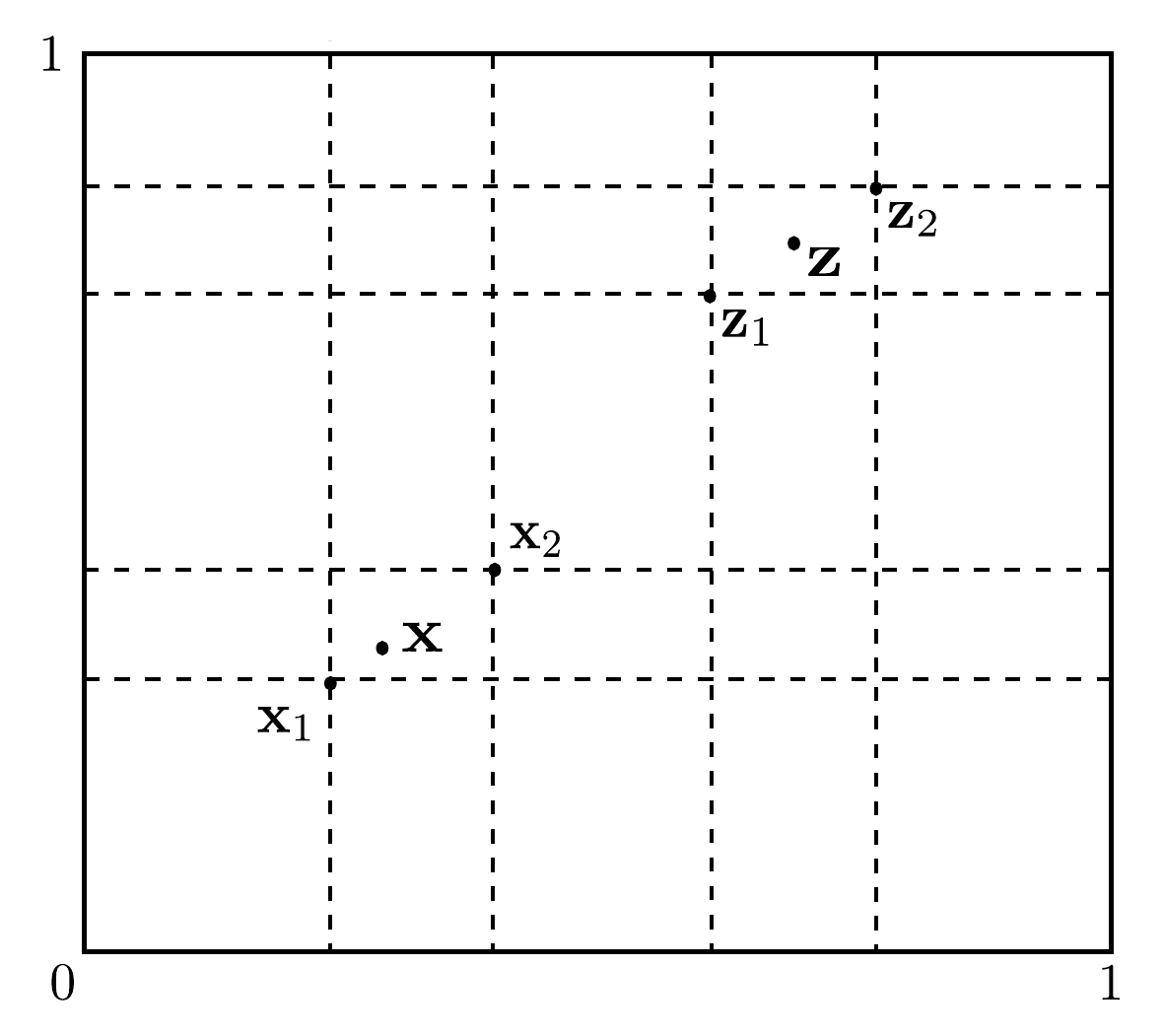}   
\end{center}
\caption{Respective positions of $\bx, \bx_1, \bx_2$ and $\bz, \bz_1, \bz_2$ with $d=2$.}
\label{figure_proof2}
\end{figure}

Thus, we have to choose the grid step $\delta$ such that 
\begin{align}
 \displaystyle \mathop{\sup_{\bx_1, \bx_2 \in [0,1]^d}}_{\|\bx_1 - \bx_2\|_{\infty} \leq \delta} \big|1 - K_n(\bx_1, \bx_2)\big| \leq \frac{\varepsilon^2}{8} .\label{inequality_proof1}
\end{align}
By {\bf Assumption 1} and the definition of the random forest grid step, there exist constants $C,A>0$ and $0<\alpha<2$ such that, for all $\varepsilon>0$, if 
\begin{align}
\delta \geq C \exp(- A/\varepsilon^{\alpha}), \label{inequality_delta}
\end{align}
then (\ref{inequality_proof1}) is satisfied. Hence, if $\delta$ satisfies (\ref{inequality_delta}), then $\tilde{\mathcal{G}}_{\delta}$ is a covering of $\varepsilon$-bracket of $\mathcal{H}$. In that case, the number $N_{[~]} (\varepsilon, \mathcal{F}, L_2(P))$ of $\varepsilon$-bracket needed to cover $\mathcal{H}$ satisfies 
\begin{align*}
N_{[~]} (\varepsilon, \mathcal{F}, L_2(P))  \leq \textrm{Card}(\tilde{\mathcal{G}}_{\delta}) \leq \textrm{Card}(\mathcal{G}_{\delta})^4 \leq  \left( \frac{1}{\delta}\right)^{4d}.
\end{align*}
Consequently, 
\begin{align*}
\sqrt{\log N_{[~]} (\varepsilon, \mathcal{F}, L_2(P))} \leq  \sqrt{\frac{2Ad}{\varepsilon^{\alpha} } - 2d \log C}
\end{align*}
where the last term is integrable near zero since $\alpha < 2$. Thus, according to Theorem $2.5.6$ in \citet{VaWe96} \citep[and the remark at the beginning of Section $2.5.2$ in][]{VaWe96} , the class $\mathcal{H}$ is $\mathds{P}_{\Theta}$-Donsker.

To conclude the proof, consider a random forest satisfying {\bf (H1)}. From above, we see that the class $\mathcal{H}$ is $\mathds{P}_{\Theta}$-Donsker. Recall that $\mathcal{F}_2 = \lbrace g_{\bx}: \theta \mapsto m_n(\bx, \theta): \bx \in$  $ [0,1]^d \rbrace$, where 
\begin{align*}
m_{n}(\bx, \Theta) =  \sum_{i=1}^n Y_i f_{\bx, \bX_i}(\Theta).
\end{align*}
Since the training set $\mathcal{D}_n$ is fixed, we have
\begin{align*}
& \sup\limits_{g_{\bx} \in \mathcal{F}_2} \big\lbrace \mathds{E}|g_{\bx}|\left(d\mathds{G}_M - d\mathds{G} \right) \big\rbrace \\
 & \quad =  \sup\limits_{\bx \in [0,1]^d} \bigg\lbrace \mathds{E}\Big|\sum_{i=1}^n Y_i f_{\bx, \bX_i}\Big|\left(d\mathds{G}_M - d\mathds{G} \right) \bigg\rbrace \\
 & \quad \leq \sum_{i=1}^n |Y_i| \sup\limits_{\bx \in [0,1]^d} \Big\lbrace \mathds{E}|f_{\bx, \bX_i}|\left(d\mathds{G}_M - d\mathds{G} \right) \Big\rbrace \\
 & \quad \leq \left( \sum_{i=1}^n |Y_i| \right) \sup\limits_{\bx, \bz \in [0,1]^d} \Big\lbrace \mathds{E}|f_{\bx, \bz}|\left(d\mathds{G}_M - d\mathds{G} \right) \Big\rbrace,
\end{align*}
which tends to zero as $M$ tends to infinity, since the class $\mathcal{H}$ is $\mathds{P}_{\Theta}$-Donsker.

Finally, note that Breiman's random forests are discrete, thus satisfying {\bf(H1)}. Uniform forests are continuous and satisfy {\bf (H1)} according to Lemma \ref{grid_step_uniform_forest}.

\end{proof}



\subsection{Proof of Theorem \ref{lemme_dependency_M_n}}

Observe that,
\begin{align*}
 & \Big(  m_{M,n}(\bX, \Theta_1, \hdots, \Theta_m)  - m(\bX) \Big)^2 \\
 & \quad =  \Big( m_{M,n}(\bX, \Theta_1, \hdots, \Theta_m) - \mathds{E}_{\Theta}\left[m_n(\bX, \Theta)\right] \Big)^2 + \Big( \mathds{E}_{\Theta} \left[m_n(\bX, \Theta)\right] - m(\bX) \Big)^2\\
  & \qquad + 2\Big( \mathds{E}_{\Theta}\left[m_n(\bX, \Theta)\right] - m(\bX) \Big)\Big( m_{M,n}(\bX, \Theta_1, \hdots, \Theta_m) - \mathds{E}_{\Theta}\left[m_n(\bX, \Theta)\right]\Big).
\end{align*}
Taking the expectation on both sides, we obtain
\begin{align*}
R(m_{M,n},m) & = R(m_{\infty,n},m) + \mathds{E} \Big[ m_{M,n}(\bX, \Theta_1, \hdots, \Theta_m) - \mathds{E}_{\Theta}\left[m_n(\bX, \Theta)\right] \Big]^2,
\end{align*}
by noticing that
\begin{align*}
& \mathds{E} \bigg[ \Big( m_{M,n}(\bX, \Theta_1, \hdots, \Theta_m) - \mathds{E}_{\Theta}\left[m_n(\bX, \Theta)\right]\Big) \Big( \mathds{E}_{\Theta}\left[m_n(\bX, \Theta)\right] - m(\bX) \Big) \bigg] \\
 & \quad = \mathds{E}_{\bX, \mathcal{D}_n} \Bigg[ \Big( \mathds{E}_{\Theta}\left[m_n(\bX, \Theta)\right] - m(\bX) \Big) \\
& \qquad \qquad \times \mathds{E}_{\Theta_1, \hdots, \Theta_M} \Big[ m_{M,n}(\bX, \Theta_1, \hdots, \Theta_m) - \mathds{E}_{\Theta}\big[m_n(\bX, \Theta)\big]\Big]\Bigg] \\
& \quad = 0,
\end{align*}
according to the definition of $m_{M,n}$. Fixing $\bX$ and $\mathcal{D}_n$, note that random variables $m_n(\bX,\Theta_1), \hdots, m_n(\bX,\Theta_1)$ are independent and identically distributed. Thus, we have
\begin{align*}
& \mathds{E} \left[ m_{M,n}(\bX, \Theta_1, \hdots, \Theta_m) - \mathds{E}_{\Theta}\left[m_n(\bX, \Theta)\right] \right]^2 \\
& = \mathds{E}_{\bX, \mathcal{D}_n} \mathds{E}_{\Theta_1, \hdots, \Theta_M} \left[ \frac{1}{M}\sum_{m=1}^M m_n(\bX, \Theta_m) - \mathds{E}_{\Theta}\left[m_n(\bX, \Theta)\right] \right]^2\\
&  = \frac{1}{M} \times \mathds{E} \Big[ \mathds{V}_{\Theta} \left[ m_n \left(\bX, \Theta \right) \right] \Big],
\end{align*}
which conludes the first part of the proof. Now, note that, 
\begin{align*}
R(m_{M,n}) - R(m_{\infty,n}) & = \frac{1}{M} \times \mathds{E} \Big[ \mathds{V}_{\Theta} \left[ m_n(\bX, \Theta) \right] \Big]\\
& = \frac{1}{M} \times \mathds{E} \left[ \mathds{V}_{\Theta} \left[ \sum_{i=1}^n W_{ni}(\bX, \Theta) (m(\bX_i) + \varepsilon_i) \right] \right]\\
& \leq \frac{1}{M} \times \left[ 8 \|m\|_{\infty}^2 + 2 \mathds{E} \left[ \mathds{V}_{\Theta} \left[ \sum_{i=1}^n W_{ni}(\bX, \Theta) \varepsilon_i \right] \right] \right] \\
& \leq \frac{1}{M} \times \left[ 8 \|m\|_{\infty}^2 + 2 \mathds{E} \left[ \max\limits_{1 \leq i \leq n } \varepsilon_i - \min\limits_{1 \leq j \leq n } \varepsilon_j \right]^2 \right] \\
& \leq \frac{1}{M} \times \left[ 8 \|m\|_{\infty}^2 + 8 \sigma^2 \mathds{E} \left[ \max\limits_{1 \leq i \leq n } \frac{\varepsilon_i}{\sigma} \right]^2 \right].
\end{align*}
The term inside the brackets is the maximum of $n$ $\chi^2$-squared distributed
random variables. Thus, for all $n \in \mathds{N}^{\star}$,
\begin{align*}
\E \left[ \max_{1\leq i \leq n} \varepsilon_i^2 \right] \leq 1 + 4 \log n,
\end{align*}
\citep[see, e.g., Chapter 1 in][]{BoLuMa13}. Therefore,
\begin{align*}
R(m_{M,n}) - R(m_{\infty,n}) & \leq \frac{8}{M} \times \big(  \|m\|_{\infty}^2 +  \sigma^2 (1 + 4 \log n) \big).
\end{align*}

\subsection{Proof of Theorem \ref{consistency_independent_forest} and Proposition \ref{proposition_consistency_uniform_cut}}

The proof of Theorem \ref{consistency_independent_forest} is based on Stone's theorem which is recalled here.

\begin{stone}

Assume that the following conditions are satisfied for every distribution of $\mathbf{{\bf X}}$:
\begin{enumerate}

\item[(i)] There is a constant $c$ such that for every non negative measurable function $f$ satisfying $\mathds{E} f(\mathbf{{\bf X}}) < \infty$ and any $n$,
\begin{align*}
\mathds{E}\left( \sum_{i=1}^n W_{ni}({\bf X}) f({\bf X}_i) \right) \leq c ~\mathbb{E} \left( f({\bf X}) \right).
\end{align*}

\item[(ii)] There is a $D>1$ such that, for all $n$, 
\begin{align*}
\mathds{P}\left( \sum_{i=1}^n W_{ni}({\bf X}) < D \right) = 1.
\end{align*}

\item[(iii)] For all $a > 0$, 
\begin{align*}
\lim\limits_{n \to \infty} \mathds{E}\left( \sum_{i=1}^n W_{ni}({\bf X}) \mathds{1}_{\|{\bf X} - {\bf X}_i\| > a} \right) = 0.
\end{align*}

\item[(iv)] The sum of weights satisfies
\begin{align*}
\sum_{i=1}^n W_{ni}({\bf X}) \underset{n \to \infty}{\to} 1  \quad \textrm{in probability}.
\end{align*}

\item[(v)] 
\begin{align*}
\lim\limits_{n \to \infty} \mathds{E}\left( \max_{1 \leq i \leq n} W_{ni}({\bf X}) \right) = 0.
\end{align*}
\end{enumerate}
Then the corresponding regression function estimate $m_n$ is universally $\mathds{L}^2$ consistent, that is, 
\begin{align*}
\lim\limits_{n \to \infty} \mathds{E} \left[ m_{\infty,n}(\bX) - m(\bX) \right]^2 = 0,
\end{align*}
for all distributions of $(\mathbf{{\bf X}}, Y)$ with $\mathds{E} Y^2 < \infty$.
\end{stone}

\begin{proof}[Proof of Theorem \ref{consistency_independent_forest}]

We check the assumptions of Stone's theorem. For every non negative measurable function $f$ satisfying $\mathds{E} f(\mathbf{{\bf X}}) < \infty$ and for any $n$, almost surely,
\begin{align*}
\mathds{E}_{\bX, \mathcal{D}_n}\left( \sum_{i=1}^n W_{ni}({\bf X}, \Theta) f({\bf X}_i) \right) \leq ~\mathbb{E}_{\bX} \left( f({\bf X}) \right), 
\end{align*}
where 
\begin{align*}
W_{ni}({\bf X}, \Theta)= \frac{\mathds{1}_{\bX_i \in A_n(\bX, \Theta)}}{N_n(\bX, \Theta)}
\end{align*}
are the weights of the random tree $\mathcal{T}_n(\Theta)$ \citep[see the proof of Theorem $4.2$ in][]{GyKoKrWa02}. Taking expectation with respect to $\Theta$ from both sides, we have
\begin{align*}
\mathds{E}_{\bX, \mathcal{D}_n}\left( \sum_{i=1}^n W_{ni}^{\infty}(\bX) f({\bf X}_i) \right) \leq ~ \mathbb{E}_{\bX} \left( f({\bf X}) \right), 
\end{align*}
which proves the first condition of Stone's theorem.

According to the definition of random forest weights $W_{ni}^{\infty}$, since $\sum_{i=1}^n W_{ni}({\bf X}, \Theta)$ $ \leq 1$ almost surely, we have 
\begin{align*}
\sum_{i=1}^n W_{ni}^{\infty}(\bX) = \mathds{E}_{\Theta} \left[ \sum_{i=1}^n W_{ni}({\bf X}, \Theta) \right] \leq 1.
\end{align*}

To check condition $(iii)$, note that, for all $a>0$, 
\begin{align*}
\mathds{E}\left[ \sum_{i=1}^n W_{ni}^{\infty}(\bX) \mathds{1}_{\|{\bf X} - {\bf X}_i\|_{\infty} > a} \right] = & \mathds{E}\left[ \sum_{i=1}^n \mathds{1}_{ \bX  \overset{\Theta}{\leftrightarrow} \bX_i}  \mathds{1}_{\|{\bf X} - {\bf X}_i\|_{\infty} > a} \right]\\
= & \mathds{E} \bigg[ \sum_{i=1}^n \mathds{1}_{\bX  \overset{\Theta}{\leftrightarrow} \bX_i} \mathds{1}_{\|{\bf X} - {\bf X}_i\|_{\infty} > a} \\
& \qquad \qquad  \times  \mathds{1}_{\textrm{diam}(A_n({\bf X},\Theta)) \geq a/2 }  \bigg],
\end{align*}
because $\mathds{1}_{\|{\bf X} - {\bf X}_i\|_{\infty} > a}  \mathds{1}_{\textrm{diam}(A_n({\bf X},\Theta)) < a/2 } = 0$. Thus, 
\begin{align*}
\mathds{E}\bigg[ \sum_{i=1}^n W_{ni}^{\infty}(\bX) \mathds{1}_{\|{\bf X} - {\bf X}_i\|_{\infty} > a} \bigg] & \leq 
\mathds{E}\bigg[ \mathds{1}_{\textrm{diam}(A_n({\bf X},\Theta)) \geq a/2 } \\
& \qquad \qquad \times \sum_{i=1}^n \mathds{1}_{\bX  \overset{\Theta}{\leftrightarrow} \bX_i} \mathds{1}_{\|{\bf X} - {\bf X}_i\|_{\infty} > a}   \bigg] \\
& \leq  ~\mathds{P} \Big[ \textrm{diam}(A_n({\bf X},\Theta)) \geq a/2 \Big],
\end{align*}
which tends to zero, as $n \to \infty$, by assumption.

To prove assumption $(iv)$, we follow the arguments developed by \cite{BiDeLu08}. For completeness, these arguments are recalled here. Let us consider the  partition associated with the random tree $\mathcal{T}_n(\Theta)$. By definition, this partition has $2^{k}$ cells, denoted by $A_1, \hdots, A_{2^{k}}$. For $1 \leq i \leq 2^{k}$, let $N_i$ be the number of points among $\bX, \bX_1, \hdots, \bX_n$ falling into $A_i$. Finally, set $\mathcal{S} = \{\bX, \bX_1, \hdots, \bX_n\}$. Since these points are independent and identically distributed, fixing the set $\mathcal{S}$ (but not the order of the points) and $\Theta$, the probability that $\bX$ falls in the $i$-th cell is $N_i/(n+1)$. Thus, for every fixed $t > 0$,
\begin{align*}
\mathds{P}\Big[ N_n(\bX, \Theta) < t  \Big] & = \mathds{E} \bigg[ \mathds{P} \Big[ N_n(\bX, \Theta) < t \Big| \mathcal{S}, \Theta \Big] \bigg]\\
 & = \mathds{E}\left[ \sum_{i: N_i < t+1} \frac{N_i}{n+1} \right]\\
 & \leq \frac{2^{k}}{n+1}t.
\end{align*}
Thus, by assumption, $N_n(\bX, \Theta) \to \infty$ in probability, as $n \to \infty$. Consequently, observe that
\begin{align*}
\sum_{i=1}^n W_{ni}^{\infty}(\bX) & = \E_{\T} \left[ \sum_{i=1}^n W_{ni}(\bX, \T) \right]\\
								& = \E_{\T} \Big[ \mathds{1}_{N_n(\bX,\T) \neq 0}  \Big] \\
								& = \mathds{P}_{\T} \left[ N_n(\bX,\T) \neq 0 \right]\\
								& \to 1 \quad \textrm{as}~n \to \infty.
\end{align*}

At last, to prove $(v)$, note that, 
\begin{align*}
\mathds{E} \left[ \max_{1 \leq i \leq n} W_{ni}^{\infty}({\bf X}) \right] & \leq  \mathds{E} \left[ \max_{1 \leq i \leq n}  \frac{\mathds{1}_{{\bf {\bf X}}_i \in A_n(\bX, \T)}}{N_n(\bX, \T)}  \right]\\
& \leq  \mathds{E} \left[ \frac{1}{N_n(\bX, \T)}  \right]\\
& \quad \to 0 \quad \textrm{as}~n \to \infty,
\end{align*}
since $N_n(\bX, \T) \to \infty$ in probability, as $n \to \infty$.
\end{proof}

\begin{proof}[Proof of Proposition \ref{proposition_consistency_uniform_cut}]
We check conditions of Theorem \ref{consistency_independent_forest}. Let us denote by $V_{nj}(\bX, \Theta)$ the length of the $j$-th side of the cell containing $\bX$ and $K_{nj}(\bX, \Theta)$ the number of times the cell containing $\bX$ is cut along the $j$-coordinate. Note that, if $U_1, \hdots, U_n$ are independent uniform on $[0,1]$, 
\begin{align*}
\mathds{E} \left[ V_{nj}(\bX, \Theta) \right] & \leq \mathds{E} \left[ \mathds{E} \left[ \prod_{l=1}^{K_{nj}(\bX, \Theta)} \max (U_i, 1 - U_i) | K_{nj}(\bX, \Theta) \right] \right]\\
& = \E \bigg[ \Big[ \E \big[ \max(U_1, 1 - U_1) \big] \Big]^{K_{nj}(\bX, \Theta)} \bigg]\\
& = \E \left[ \left(\frac{3}{4} \right)^{K_{nj}(\bX, \Theta)} \right].
\end{align*}
Since $K_{nj}(\bX, \Theta)$ is distributed as a binomial $\mathcal{B}(k_n, 1/d)$, $K_{nj}(\bX, \Theta) \to + \infty$ in probability, as $n$ tends to infinity. Thus $\mathds{E} \left[ V_{nj}(\bX, \Theta) \right] \to 0$ as $n \to \infty$.
\end{proof}

\subsection{Proof of Theorem \ref{theoreme_convergence_quantile_forest}}

To prove Theorem \ref{theoreme_convergence_quantile_forest}, we need the following lemma which states that the cell diameter of a quantile tree tends to zero.

\begin{lemme}
\label{diametre_arbre_quantile}
Assume that $\bX$ has a density $f$ over $[0,1]^d$, with respect to the Lebesgue measure and that there exist two constants $c, C > 0$ such that, for all $\bx \in [0,1]^d$,
\begin{align*}
c \leq f(\bx) \leq C.
\end{align*}
Thus, for all $q \in [1/2,1)$, the $q$ quantile tree defined in {\bf Algorithm 2} satisfies, for all $\gamma$,
\begin{align*}
\mathds{P}_{\bX, \Theta} \big[ \textrm{diam}(A_n({\bf X}, \Theta)) > \gamma \big] \underset{n \to \infty}{\to} 0.
\end{align*}
\end{lemme}

\begin{proof}[Proof of Lemma \ref{diametre_arbre_quantile}]

Set $q\in [1/2,1)$. At first, consider a theoretical $q$ quantile tree where cuts are made similarly as in the $q$ quantile tree but by selecting $q_n\in [1-q,q]$ and by performing the cut at the  $q_n$ theoretical quantile (instead of empirical one). The tree is then stopped at level $k$, where $k \in \mathds{N}$ is a parameter to be chosen later. 
Then, consider a cell $A = \prod_{j=1}^d [a_i, b_i]$ of the theoretical $q$ quantile tree. Assume that this cell is cut along the first coordinate and let $z$ be the split position. Thus, by definition of the theoretical $q$ quantile tree, there exists $q'\in [1-q, q]$ such that 
\begin{align*}
\int_{a_1}^z \int_{a_2}^{b_2} \hdots \int_{a_d}^{b_d}  f(\bx) \textrm{d}x_1  \hdots \textrm{d}x_d   = q' \int_{a_1}^{b_1} \hdots \int_{a_d}^{b_d}  f(\bx) \textrm{d}x_1 \hdots \textrm{d}x_d,
\end{align*}
that is
\begin{align}
(1 - q' )\int_{a_1}^z g(x_1) \textrm{d}x_1 = q' \int_z^{b_1} g(x_1) \textrm{d}x_1,\label{equation_proof_1}
\end{align}
where $g(x_1) = \int_{a_2}^{b_2} \hdots \int_{a_d}^{b_d}  f(\bx) \textrm{d}x_2 \hdots \textrm{d}x_d$. Letting $\mu_{d-1}(A) = \prod_{j=2}^d (b_j - a_j)$, by assumption, we have, for all $x_1\in [a_1, b_1]$, 
\begin{align*}
c \mu_{d-1}(A) \leq g(x_1) \leq C \mu_{d-1}(A).
\end{align*}
Hence, using (\ref{equation_proof_1}), we obtain
\begin{align*}
(1 - q') (z-a_1) c \mu_{d-1}(A) \leq  q' \int_z^{b_1} g(x_1) \textrm{d}x_1 \leq  q' (b_1-z) C \mu_{d-1}(A), 
\end{align*}
which leads to 
\begin{align}
\frac{z-a_1}{b_1-a_1} \leq \frac{q' C}{q' C + (1-q')c} \in ]0,1[. \label{arbre_quantile_proof_1}
\end{align}
Similarly,  
\begin{align*}
q' (b_1-z) c \mu_{d-1}(A) \leq  (1 - q' )\int_{a_1}^zg(x_1) \textrm{d}x_1\leq  (1-q') (z-a_1) C \mu_{d-1}(A),
\end{align*}
which yields
\begin{align}
\frac{b_1-z}{b_1-a_1} \leq \frac{C(1-q')}{q' c + (1-q')C} \in ]0,1[. \label{arbre_quantile_proof_2}
\end{align}
Combining (\ref{arbre_quantile_proof_1}) and (\ref{arbre_quantile_proof_2}), we deduce that 
\begin{align*}
\max\left( \frac{z-a_1}{b_1-a_1}, 1 - \frac{z-a_1}{b_1-a_1} \right) & \leq \max\left(\frac{q' C}{q' C + (1-q')c}, \frac{C(1-q')}{q' c + (1-q')C} \right)\\
& \leq \frac{Cq}{(1-q)(c+C)}.
\end{align*}
Consequently, letting 
\begin{align*}
\alpha = \frac{Cq}{(1-q)(c+C)} \in (0,1),
\end{align*}
the first dimension of $A$ is reduced at most by a factor $1-\alpha>0$ and at least by a factor $\alpha<1$.

Denote by $V_{ik}(\bX, \Theta)$ the length of the $i$-th side of the cell containing $\bX$ at level $k$ in the theoretical $q$ quantile tree, and let $K_{ik}(\bX, \Theta)$ the number of times this cell has been cut along the $i$-th coordinate. Hence, for all $i \in \{1,\hdots, d\}$,
\begin{align}
\E \left[ V_{ik}(\bX, \Theta) \right] \leq \E \big[\alpha^{K_{ik}(\bX, \Theta)} \big]  \underset{k \to \infty}{\to} 0, \label{eq_diam_0_arbre_quantile_theo}
\end{align}
which proves that the cell diameter of the theoretical $q$ quantile tree tends to zero, as the level $k$ tends to infinity.

Now, consider the empirical $q$ quantile tree as defined in {\bf Algorithm 2} but stopped at level $k$. Thus, for $n$ large enough, at each step of the algorithm, $q_n$ is selected in $[1-q,q]$. Set $\varepsilon, \eta > 0$ and let $G_k(\bX, \Theta)$ be the event where all $k$ cuts used to build the cell $A_k(\bX, \Theta)$ are distant of less than $\eta$ from cuts used to build the cell $A^{\star}_k(\bX, \Theta)$ of the theoretical $q$ quantile tree. Thus, 
\begin{align}
\E \big[ \textrm{diam}(A_k(\bX, \T)) \big] & = \E \big[ \textrm{diam}(A_k(\bX, \T)\mathds{1}_{G_k(\bX, \Theta)}) \big] \nonumber \\
& \quad  + \E \big[ \textrm{diam}(A_k(\bX, \T))\mathds{1}_{G_k(\bX, \Theta)^c}) \big] \nonumber \\
& \leq \E \big[ \textrm{diam}(A_k(\bX, \T)\mathds{1}_{G_k(\bX, \Theta)}) \big] + \P \big[ G_k(\bX, \Theta)^c \big]. \label{final_proof_1}
\end{align}
From equation (\ref{eq_diam_0_arbre_quantile_theo}), there exists $k_0 \in \mathds{N}^{\star}$ such that 
\begin{align*}
\mathds{E} \big[ \textrm{diam}(A^{\star}_{k_0}({\bf X}, \Theta)) \big] < \varepsilon.
\end{align*}
Since, on the event $G_k(\bX, \Theta)$, the $k$ consecutive cuts used to build $A^{\star}_k(\bX, \Theta)$ are distant of less than $\eta$ from the $k$ cuts used to design $A_k(\bX, \Theta)$, we have 
\begin{align}
\mathds{E} \big[\textrm{diam}(A_{k_0}({\bf X}, \Theta))\big] < k_0\eta + \varepsilon. \label{final_proof_2}
\end{align}

With respect to the second term in equation (\ref{final_proof_1}),  consider a cell $A$ of the empirical $q$ quantile tree. Without loss of generality, we assume that the next split is performed along the first coordinate. Let $F^A$ (resp. $F_{n}^A$) be the one dimensional conditional  distribution function (resp. empirical distribution function) of $\bX$ given that $\bX \in A$. Denote by $z^A_{n}$ the position of the empirical split performed in $A$. Since $\bX$ is uniformly distributed over $[0,1]^d$, $(F^A)'\geq c/ \mu(A)$. Thus, since $F^A$ is an increasing function, if $\sup\limits_{x\in A} |F^A(x) - F^A_{n}(x) | \leq c \eta/ \mu(A)$ then 
\begin{align*}
\inf\limits_{z \in \mathcal{Z}^A} |z^A_{n} - z| \leq \eta,
\end{align*}
where $\mathcal{Z}^A = \{ z, F^A(z) \in [1-q, q] \}$. Recall that $A_1(\bX, \T),\hdots, A_k(\bX, \T)$ are the consecutive cells containing $\bX$ designed with $\Theta$. Observe that, conditionally on the position of the split, data on the left side of the split are still independent and identically distributed according to Proposition $2.1$ in \citet{BiCeGu12}. Thus, we have 
\begin{align*}
\P \big[ G_k(\bX, \Theta)^c \big] 
& \leq \sum_{\ell = 1}^k \P \bigg[ \inf\limits_{z \in \mathcal{Z}^{A_{\ell}(\bX, \T)}} |z^{A_{\ell}(\bX, \T)}_{n} - z^{A_{\ell}(\bX, \T)}| > \eta  \bigg]\\
& \leq 	\sum_{\ell = 1}^k  \mathds{E} \left[ \P \bigg[ \inf\limits_{z \in \mathcal{Z}^{A_{\ell}(\bX, \T)}} |z^{A_{\ell}(\bX, \T)}_{n} - z^{A_{\ell}(\bX, \T)}| > \eta \bigg| N(A_{\ell}(\bX, \T)) \bigg] \right]\\
& \leq 	\sum_{\ell = 1}^k  \mathds{E} \bigg[ \P \bigg[ \sup\limits_{x} |F(x) - F_{n}(x) | \geq \frac{c \eta}{\mu(A)}  \bigg| N(A_{\ell}(\bX, \T)),\\
& \qquad \qquad  \qquad \qquad A_{\ell}(\bX, \T) \bigg] \bigg].
\end{align*}
Consequently, 
\begin{align*}
\P \big[ G_k(\bX, \Theta)^c \big]  
& \leq 2 \sum_{\ell = 1}^k \mathds{E} \left[ \exp\left(-\frac{2N(A_{\ell}(\bX, \T))c^2 \eta^2}{\mu(A_{\ell}(\bX, \T))^2}\right)\right]\\
& \qquad \textrm{\citep[see][]{Ma90}} \\
& \leq 2 k \exp\Big(-2 c^2 \eta^2 \big((1-q)^k n - \frac{1}{1-q}\big) \Big) \\
& \qquad \Big(\textrm{since}~ \min_{\ell} N(A_{\ell}(\bX, \T)) \geq (1-q)^k n - \frac{1}{1-q}\Big).
\end{align*}
Thus, for all $k$ and for all $n>\log k/(2 c^2 \eta^2 (1-q)^k)$, we obtain 
\begin{align}
\P \big[ G_k(\bX, \Theta)^c \big] 
& \leq \eta. \label{final_proof_3}
\end{align}
Gathering (\ref{final_proof_1}), (\ref{final_proof_2}) and (\ref{final_proof_3}), we conclude that, for all $n>\log k_0/(2 c^2 \eta^2 (1-q)^{k_0})$, 
\begin{align*}
\E \big[ \textrm{diam}(A_k(\bX, \T)) \big] \leq (k_0+1) \eta + \varepsilon.
\end{align*}
Since the diameter is a non increasing function of the level of the tree, the cell diameter of the fully developed tree (which contain exactly one point in each leaf) is lower than that of the tree stopped at $k_0$. Letting $A(\bX, \Theta)$ the cell of the (fully developed) empirical $q$ quantile tree (defined in {\bf Algorithm 2}), we have 
\begin{align*}
\E \big[ \textrm{diam}(A_n(\bX, \Theta)) \big] & \leq  \E \big[ \textrm{diam}(A_k(\bX, \T)) \big] \\
									 & \leq  (k_0+1) \eta + \varepsilon,
\end{align*}
which concludes the proof, since $\varepsilon$ and $\eta$ can be made arbitrarily small.

\end{proof}

\begin{proof}[Proof of Theorem \ref{theoreme_convergence_quantile_forest}]
We check the conditions of Stone's theorem. 
Condition $(i)$ is satisfied since the regression function is uniformly continuous and $\textrm{Var}[Y|\bX] = \sigma^2$ \citep[see remark after Stone theorem in][]{GyKoKrWa02}.

Condition $(ii)$ is always satisfied for random trees. Condition $(iii)$ is verified since 
\begin{align*}
\mathds{P}_{\bX, \Theta} \left[ \textrm{diam}(A_n({\bf X}, \Theta)) > \gamma \right] \underset{n \to \infty}{\to} 0,
\end{align*}
according to Lemma \ref{diametre_arbre_quantile}.

Since each cell contains exactly one data point, 
\begin{align*}
    \sum_{i=1}^n W_{ni}(x) = & \sum_{i=1}^n \mathds{E}_{\Theta}\left[ \frac{\mathds{1}_{{\bf {\bf X}}_i \in A_n(\bX, \T)}}{N_n(\bX, \T)}  \right] \\
= & \mathds{E}_{\Theta} \left[ \frac{1}{N_n(\bX, \T)} \sum_{i=1}^n \mathds{1}_{{\bf {\bf X}}_i \in A_n(\bX, \T)}\right]\\
= & 1.
\end{align*}
Thus, conditions $(iv)$ of Stone theorem is satisfied.

To check $(v)$, observe that in the subsampling step, there are exactly $\binom{a_n-1}{n-1}$ choices to pick a fixed observation $\bX_i$. Since $\bx$ and $\bX_i$ belong to the same cell only if $\bX_i$ is selected in the subsampling step, we see that
\begin{align*}
 \mathds{P}_{\Theta}\left[ \bX  \overset{\Theta}{\leftrightarrow} \bX_i \right] 
\leq & \frac{\binom{a_n-1}{n-1}}{\binom{a_n}{n}} = \frac{a_n}{n}.
\end{align*} 
So, 
\begin{align*}
\mathds{E}\left[\max\limits_{1 \leq i \leq n} W_{ni}(\bX)\right] & \leq  \mathds{E}\left[ \max_{1 \leq i \leq n } \mathds{P}_{\Theta} \left[ \bX  \overset{\Theta}{\leftrightarrow} \bX_i \right] \right] \leq \frac{a_n}{n}, 
\end{align*}
which tends to zero by assumption.
\end{proof}

\subsection{Proofs of Technical Lemmas \ref{lemme_connection_probability_uniform_forest_dim1} and \ref{tech_lemma_1}}

\begin{techlemme}\label{tech_lemma_1}
Take $k \in \mathds{N}$ and consider a uniform random forest where each tree is stopped at level $k$. For all $\bx, \bz \in [0,1]^d$, its connection function satisfies
\begin{align*}
K_k(0,|\bx-\bz|) \leq K_k(\bx, \bz),
\end{align*}
where $|\bx-\bz| = (|x_1-z_1|, \hdots, |x_d-z_d|)$.
\end{techlemme}

\begin{proof}
Take $x, z \in [0,1]$. Without loss of generality, one can assume that $x<z$ and let $\mu = z-x$. Consider the following two configurations.
\begin{figure}[h!]
\begin{center}
\includegraphics[scale=1]{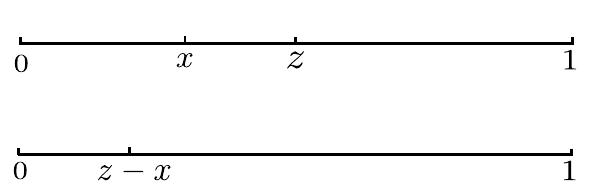}
\caption{Scheme of configuration $1$ (at the top) and $2$ (at the bottom).}
\end{center}
\end{figure}

For any $k \in \mathds{N}^{\star}$, we let ${\bf d}_k = (d_1, \hdots, d_k)$ (resp. ${\bf d}'_k = (d'_1, \hdots, d'_k)$) be $k$ consecutive cuts in configuration $1$ (resp. in configuration $2$). We denote by $\mathcal{A}_k$ (resp. $\mathcal{A}'_k$) the set where ${\bf d}_k$ (resp. ${\bf d}'_k$) belong.

We show that for all $ k \in \mathds{N}^ {\star}$, there exists a coupling between $\mathcal{A}_k$ and $\mathcal{A}'_k$ satisfying the following property: any $k$-tuple ${\bf d}_k$ is associated with a $k$-tuple ${\bf d}'_k$ such that 
\begin{enumerate}
\item if ${\bf d}_k$ separates $[x,z]$ then ${\bf d}'_k$ separates $[0,z-x]$,
\item if ${\bf d}_k$ does not separate $[x,z]$ and ${\bf d}'_k$ does not separate $[0,z-x]$, then the length of the cell built with ${\bf d}_k$ is higher than the one built with ${\bf d}'_k$.
\end{enumerate}
We call $\mathcal{H}_k$ this property. We now proceed by induction. For $k=1$, we use the function $g$ to map $\mathcal{A}_1$ into $\mathcal{A}'_1$ such that:
\begin{align*}
g_1(u) = \left\lbrace
\begin{array}{ll}
u & \textrm{if}~ u >z  \\
z-u & \textrm{if}~ u \leq z
\end{array}
\right.
\end{align*}
Thus, for any $d_1 \in \mathcal{A}_1$, if $d_1$ separates $[x,z]$, then $d_1' = g_1(d_1)$ separates $[0,z-x]$. Besides, the length of the cell containing $[x,z]$ designed with the cut $d_1$ is higher than that of the cell containing $[0,z-x]$ designed with the cut $d_1'$. Consequently, $\mathcal{H}_1$ is true.

Now, take $k>1$ and assume that $\mathcal{H}_k$ is true. Consequently, if ${\bf d}_k$ separates $[x,z]$ then $g_k({\bf d}_k)$ separates $[0, z-x]$. In that case, ${\bf d}_{k+1}$ separates $[x,z]$ and $g_{k+1}({\bf d}_{k+1})$ separates $[0, z-x]$. 
Thus, in the rest of the proof, we assume that ${\bf d}_k$ does not separate $[x,z]$ and $g_k({\bf d}_k)$ does not separate $[0,z-x]$. 
Let $[a_k, b_k]$ be the cell containing $[x,z]$ built with cuts ${\bf d}_k$. Since the problem is invariant by translation, we assume, without loss of generality, that $[a_k,b_k] = [0,\delta_k]$, where $\delta_k = b_k - a_k$ and $[x,z]= [x_k, x_k + \mu]$ (see Figure \ref{last_figure7}).
\begin{figure}[h!]
\begin{center}
\includegraphics[scale=1]{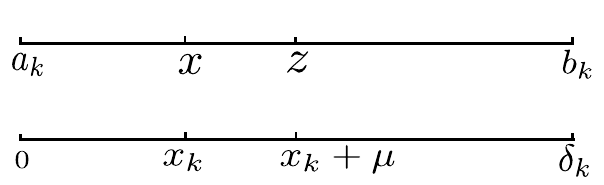}
\caption{Configuration 1a (at the top) and 1b (at the bottom).}
\label{last_figure7}
\end{center}
\end{figure}

In addition, according to $\mathcal{H}_k$, the length of the cell built with ${\bf d}_k$ is higher than the one built with ${\bf d}'_k$. Thus, one can find $\lambda \in (0,1)$ such that $d_k' = \lambda \delta_k$. This is summarized in Figure \ref{last_figure}.
\begin{figure}[h!]
\begin{center}
\includegraphics[scale=1]{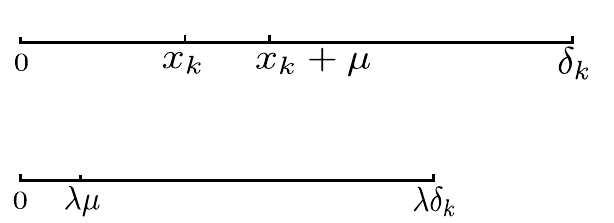}
\caption{Configuration 1b (at the top) and 2b (at the bottom).}
\label{last_figure}
\end{center}
\end{figure}

Thus, one can map $[0, \delta_k]$ into $[0, \lambda \delta_k]$ with $g_{k+1}$ defined as
\begin{align*}
g_{k+1}(u) = \left\lbrace
\begin{array}{ll}
\lambda u & \textrm{if}~ u >x_k + \mu  \\
\lambda (x + \mu - u) & \textrm{if}~ u \leq x_k + \mu
\end{array}
\right.
\end{align*}
Note that, for all $d_{k+1}$, the length of the cell containing $[x_k, x_k+\mu]$ designed with the cut $d_{k+1}$ (configuration 1b) is bigger than the length of the cell containing $[0,\mu]$ designed with the cut $d_{k+1}'=g_{k+1}(d_{k+1})$ (configuration 2b). Besides, if $d_{k+1}\in [x_k, x_k+\mu]$ then $g_{k+1}(d_{k+1}) \in [0, \mu]$. Consequently, the set of functions $g_1, \hdots, g_{k+1}$ induce a mapping of $\mathcal{A}_{k+1}$ into $\mathcal{A}'_{k+1}$ such that $\mathcal{H}_{k+1}$ holds. Thus, Technical Lemma \ref{tech_lemma_1} holds for $d=1$. 

To address the case where $d>1$, note that 
\begin{align*}
K_{k}(\bx,\bz) & = \sum\limits_{\substack{k_{1},\hdots,k_{d} \\ \sum_{j=1}^d k_{j} = k }} \frac{k!}{k_{1}! \hdots k_{d} !} \left( \frac{1}{d}\right)^k  \prod_{m=1}^d K_{k_m}(x_m, z_m)\\
& \geq \sum\limits_{\substack{k_{1},\hdots,k_{d} \\ \sum_{j=1}^d k_{j} = k }} \frac{k!}{k_{1}! \hdots k_{d} !}  \left( \frac{1}{d}\right)^k \prod_{m=1}^d  K_{k_m}(0, |z_m-x_m|)\\
& \geq K_k(0,|\bz - \bx|),
\end{align*}
which concludes the proof.

\end{proof}

\begin{techlemme} \label{lemme_connection_probability_uniform_forest_dim1}
Take $k \in \mathds{N}$ and consider a uniform random forest where each tree is stopped at level $k$. For all $x \in [0,1]$, its connection function $K_k(0,x)$ satisfies 
\begin{align*}
K_k(0,x) & = 1 - x \sum_{j=0}^{k-1} \frac{(- \ln x)^j}{j!},
\end{align*}
with the notational convention that the last sum is zero if $k=0$.
\end{techlemme}

\begin{proof}[Proof of Technical Lemma \ref{lemme_connection_probability_uniform_forest_dim1}]

The result is clear for $k=0$. Thus, set $k \in \mathds{N}^{\star}$ and consider a uniform random forest where each tree is stopped at level $k$. Since the result is clear for $x=0$, take $x \in ]0,1]$ and let $I = [0,x]$.  Thus
\begin{align*}
K_k(0,x) & = \P\left[ 0 \underset{k ~\textrm{cuts}}{\overset{\T}{\leftrightarrow}} x\right]  \\
& = \int_{z_1 \notin I} \int_{z_2 \notin I} \hdots \int_{z_k \notin I} p(\textrm{d}z_k | z_{k-1}) p(\textrm{d}z_{k-1} | z_{k-2}) \hdots p(\textrm{d}z_2|z_1) p(\textrm{d}z_1),
\end{align*}
where $z_1, \hdots, z_k$ are the positions of the $k$ cuts (see Figure \ref{figure_lemme_connection_uniforme}).

\vspace{0.5cm}
\begin{figure}[h!]
\begin{center}
\begin{tabular}{c}
\includegraphics[scale=0.6]{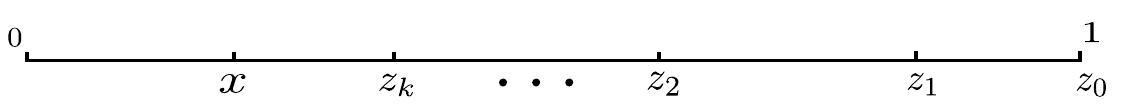} 
\end{tabular}
\end{center}
\caption{Positions of cuts $z_1, \hdots, z_k$ and $x$ with $d=1$}
\label{figure_lemme_connection_uniforme}
\end{figure}
\vspace{0.5cm}

We prove by induction that, for every integer $p$,
\begin{align*}
\int_{z_{k-p} \notin I} \hdots  \int_{z_k \notin I}& p(\textrm{d}z_k | z_{k-1}) \hdots p(\textrm{d}z_{k-p}|z_{k-p-1}) \\
& = 1 - \frac{x}{z_{k-p-1}} \left( \sum_{j=0}^{p} \frac{\left[\ln(z_{k-p-1}/x)\right]^j}{j!} \right).
\end{align*}
Denote by $\mathcal{H}_p$ this property. Since, given $z_{k-1}$, $z_k$ is uniformly distributed over $[0, z_{k-1}]$, we have
\begin{align*}
\int_{z_k \notin I} p(\textrm{d}z_k | z_{k-1}) & = 1 - \frac{x}{z_{k-1}}. 
\end{align*}
Thus $\mathcal{H}_0$ is true. Now, fix $p>0$ and assume that $\mathcal{H}_p$ is true. Let $u = z_{k-p-1}/x$. Thus, integrating both sides of $\mathcal{H}_p$, we deduce,
\begin{align*}
& \int_{z_{k-p-1}\notin I} \int_{z_{k-p} \notin I} \hdots \int_{z_k \notin I} p(\textrm{d}z_k | z_{k-1}) \hdots p(\textrm{d}z_{k-p}|z_{k-p-1}) p(\textrm{d}z_{k-p-1}|z_{k-p-2}) \\
= & \int_{z_{k-p-1}\notin I} \left[1 - \frac{x}{z_{k-p-1}} \left( \sum_{j=0}^{p} \frac{\left[\ln(z_{k-p-1}/x)\right]^j}{j!} \right) \right] p(\textrm{d}z_{k-p-1}|z_{k-p-2}) \\
= &  \int_{x}^{z_{k-p-2}} \left[ 1 - \frac{x}{z_{k-p-1}} \left( \sum_{j=0}^{p} \frac{\left[\ln(z_{k-p-1}/x)\right]^j}{j!} \right) \right] \frac{\textrm{d}z_{k-p-1}}{z_{k-p-2}} \\
= & \frac{x}{z_{k-p-2}} \int_{1}^{z_{k-p-2}/x} \left[ 1 - \frac{1}{u} \left( \sum_{j=0}^{p} \frac{\left[\ln(u)\right]^j}{j!} \right) \right]  \textrm{d}u.
\end{align*}
Using integration by parts on the last term, we conclude that $\mathcal{H}_{p+1}$ is true. Thus, for all $p>0$, $\mathcal{H}_p$ is verified. Finally, using  $\mathcal{H}_{k-1}$ and the fact that $z_0 = 1$, we conclude the proof.
\end{proof}

\bibliography{biblio-sbv}

\end{document}